\documentclass[11pt,leqno]{amsart}

\usepackage[all,cmtip]{xy}
\usepackage{a4,latexsym}
\usepackage[utf8]{inputenc}
\usepackage{amsfonts}
\usepackage[hidelinks]{hyperref}
\usepackage{amsxtra,amsmath,amsfonts,amscd, amssymb, mathrsfs, amsthm}
\usepackage{color, mathtools}
\usepackage{bbm}
\usepackage{enumitem}
\usepackage{xcolor}
\usepackage{graphicx}
\usepackage{tikz-cd}
\usepackage{braket}
\usepackage[all]{xypic}
\usepackage[toc,page,title,titletoc,header]{appendix}
\usepackage[capitalize]{cleveref}

\numberwithin{paragraph}{section}

\setlist[enumerate]{label=\it{(\roman*)},
	ref=\it{(\roman*)}}

\newcommand{\Z}{{\mathbb Z}}

\renewcommand{\Im}{\mathrm{Im}}

\newcommand{\Frac}{\mathrm{Frac}}

\newcommand{\Spec}{\mathrm{Spec}}

\newcommand{\Supp}{\mathrm{Supp}}
\newcommand{\CH}{\mathrm{CH}}
\newcommand{\Div}{\mathrm{Div}}
\newcommand{\CaCl}{\mathrm{CaCl}}
\newcommand{\Pic}{\mathrm{Pic}}

\newcommand{\OO}{\mathcal{O}}

\newcommand{\Cl}{\mathrm{Cl}}

\newcommand{\id}{\operatorname{id}\nolimits}

\newcommand{\Ob}{\operatorname{Ob}\nolimits}

\newcommand{\Hom}{\operatorname{Hom}\nolimits}
\renewcommand{\Im}{\operatorname{Im}\nolimits}
\newcommand{\Ker}{\operatorname{Ker}\nolimits}

\renewcommand{\dim}{\operatorname{dim}\nolimits}
\newcommand{\codim}{\operatorname{codim}\nolimits}

\newcommand{\zar}{\text{zar}}

\newcommand{\an}{\mathrm{an}}

\newcommand{\Cor}{\mathrm{Cor}}

\newtheorem{theorem}{Theorem}[section]

\newtheorem{corollary}[theorem]{Corollary}

\newtheorem{lemma}[theorem]{Lemma}
\newtheorem{lemma*}{Lemma}
\newtheorem{proposition}[theorem]{Proposition}

\theoremstyle{definition}
\newtheorem{definition}[theorem]{Definition}
\newtheorem{proposition&definition}[theorem]{Proposition\&Definition}
\newtheorem{lemma&definition}[theorem]{Lemma\&Definition}
\newtheorem{theorem&definition}[theorem]{Theorem\&Definition}
\newtheorem{example}[theorem]{Example}

\newtheorem{example*}{Example}
\newtheorem{remark}[theorem]{Remark}

\newtheorem{question*}{Question}

\newtheorem{art}[theorem]{}

\begin{document}
	
	\title[Intersection theory on non-archimedean analytic spaces]{Intersection theory on non-archimedean analytic spaces}

	\author[Y.~Cai]{Yulin Cai}
	\address{Westlake University, 
		Dunyu Road 600, Xihu District
		310024, Hangzhou, China}
	\email{ylcai5388339@gmail.com}

	\begin{abstract}
	We develop the intersection theory of non-archimedean analytic spaces and prove the projection formula and the GAGA principle. As an application, we naturally define the category of finite correspondences of analytic spaces.
	\end{abstract}
	

\maketitle

\setcounter{tocdepth}{1}

\tableofcontents

	
	\section{Introduction}
	
	The intersection theory of non-archimedean analytic spaces has been studied in \cite[Section~2]{gubler1998local} and \cite[Section~2.2]{ayoub2015motifs}, and the author believes that some experts have concrete idea about such a theory. 
	
	In \cite{gubler1998local}, Gubler considers the Cartier divisors on rigid analytic spaces and formal schemes, and define their intersection with irreducible analytic subsets. This theory allows him to define the local height of subvarieties over non-archimedean fields.   
	
	In \cite{ayoub2015motifs}, Ayoub develops the theory of motives on rigid analytic spaces using homotopy theory. He uses the presheaves on the category of affinoid spaces to construct the category of finite correspondence (for rigid analytic space) $\mathbf{RigCor}(K)$. Such construction avoids the intersection theory of analytic spaces.

	In this paper, we develop the intersection theory of non-archimedean analytic spaces following the idea similar to the case of algebraic varieties. We will show the flat base change formula, the projection formula and the GAGA principle to relate the intersection theories of analytic spaces and of algebraic varieties. As an application, we give a direct construction of $\mathbf{RigCor}(K)$ like \cite[Lecture~1]{mazza2006lecture} does, see \cref{equivalence of category of finite correspondences} for precise statement. In fact, we can define the higher Chow groups of analytic spaces as \cite{bloch1986algebraic} for algebraic varieties, and this definition is different from Ayoub's in \cite[Introduction g\'en\'erale]{ayoub2015motifs}.
	
	In \cref{preliminary}, we give some basic notion in the theory of Berkovich spaces, e.g. support of a coherent sheaf, Zariski image and codimension. We also extend \cite[Proposition~4.12]{ducros2009les} into an abstract form, i.e. \cref{key lemma for a local property to be global on irreducible subsets} which is a key lemma for this paper. With this lemma, we can solve the compatibility problems in our theory, e.g. see \cref{multiplicity of a coherent sheaf} and \cref{multiplicities of proper intersection}.
	
	In \cref{Meromorphic functions and Cartier divisors}, we define and study the Cartier divisors on an analytic space $X$, which form a group $\Div(X)$. The group of divisors up to linear equivalence is denoted by $\CaCl(X)$. As in the theory of schemes, we have an injective homomorphism $\CaCl(X)\hookrightarrow \Pic(X)$, and it is an isomorphism if $X$ is affinoid and reduced.
	
	In \cref{Cycles}, we give the notion of cycles, and associate a coherent sheaf with a cycle.
In particular, we can associate a closed subspace with a cycle. 
As in the theory of algebraic varieties, the flat pull-backs and proper push-forwards of cycles are defined. We prove the following flat base change formula. 
\begin{proposition}[\cref{prop:flatbasechangeofcycles}]
	Let 
	\[\xymatrix{Y'\ar[r]^{g'}\ar[d]_{f'}& Y \ar[d]^{f}\\
		X'\ar[r]_{g}& X}\]
	be a Cartesian diagram of separated, strictly $K$-analytic spaces. Assume that $f$ is proper, and that $g$ is flat, has of relative dimension $r$. Then $f'$ is proper, $g'$ is flat and has relative dimension $r$. Moreover, $g^*\circ f_*= f'_*\circ g'^*$ on $Z^*(Y)$.
\end{proposition}
	
	In \cref{properintersection}, we define intersection product of proper intersection. We will give two definitions, meaning a local one using the scheme theory and a global one using $\mathrm{Tor}$ formula. 
	For a flat morphism $f\colon Y\to X$ of $K$-analytic spaces of pure dimension, the pull-back $f^*\colon Z^*(X)\to Z^*(Y)$ preserves intersection product.

Since we have the flat pull-backs, proper push-forwards and intersection products, the expected projection formula is proved in \cref{projectionformula}.
\begin{theorem}[Projection formula]
	Let $f\colon Y\to X$ be a flat, proper morphism of regular, separated, strictly $K$-analytic spaces. Let $\alpha\in Z^*(Y)$ and $\beta\in Z^*(X)$. Assume that $\alpha$ and $f^*\beta$ intersect properly. Then $f_*(\alpha)$ and $\beta$ intersect properly and 
	\[f_*(\alpha)\cdot\beta=f_*(\alpha\cdot f^*\beta).\]
\end{theorem}

In \cref{GAGA}, we compare the intersection theories of algebraic varieties and of non-archimedean analytic spaces. We prove the GAGA principle, i.e. \cref{prop:GAGA}.

In \cref{The category of finite correspondences}, we define the category of finite correspondence $\mathrm{Cor}_K$. This category is also defined by Ayoub \cite{ayoub2015motifs} using another definition.  

\subsection*{Notation and terminology}

Throughout this paper, we fix a complete non-archimedean field $K$ with a non-trivial valuation. For a $K$-analytic space, we mean a Berkovich space over $K$, see \cite[Definition~1.2.3]{berkovich1993etale}. The structure sheaf on a $K$-analytic space $X$ with respect to the G-topology  is denoted by $\OO_X$. If it is necessary, we will use the notation $X_G$ for the G-topology instead of the ordinary topology on $X$. The ($K$-analytic) dimension of $X$ is denoted by $\dim_KX$, or $\dim X$ when there is no confusion with the fields. 

Given a point $x\in X$, $\mathscr{H}(x)$ denotes its complete residue field and $\dim_xX$ denotes the local dimension of $X$ at $x$. 

We shall simply say "coherent sheaf on $X$" for "coherent $\OO_X$-module (with respect to G-topology)", and denote $\mathcal{C}oh(X)$ the category of coherent sheaves on $X$. We also denote $\Pic(X)$ the group of isomorphism classes of invertible sheaves on $X$. Assume that $X$ is good, let $\mathcal{F}$ be a coherent sheaf on $X$ and $x\in X$. We denote by $\mathcal{F}_x$ the stalk at $x$ of $\mathcal{F}$ viewed as a sheaf of the underlying ordinary topology of $X$, i.e.
\[\mathcal{F}_x\coloneqq\lim\limits_{\substack{\longrightarrow \\ U}}\mathcal{F}(U)=\lim\limits_{\substack{\longrightarrow \\ V}}\mathcal{F}(V).\]
where $U$ runs through open neighborhoods of $x$, and $V$ runs through affinoid neighborhoods of $x$.

We will write ${\mathrm{Irr}(X)}$ for the set of all irreducible components of $X$, see \cite[\S~4.3.2]{ducros2009les}, and write $\overline{\mathrm{Irr}(X)}$ for the set of all irreducible Zariski-closed subsets of $X$. Notice that $\overline{\mathrm{Irr}(X)}$ has a partial order: $W\leq Z$ if $W\subset Z$.


By an algebraic variety over $K$, we mean a separated scheme of finite type over $K$. 

For a commutative ring $A$, $R(A)$ denotes the set of all regular elements of $A$ and $\Frac(A)=R(A)^{-1}A$, the maximal localization containing $A$ as a subring. For an $A$-module $M$, we denote $\mathrm{Ann}(M)\coloneq\{a\in A\mid a\cdot M=0\}$.
\section{Preliminary}

\label{preliminary}

For the convenience of the reader and further uses, in the section, we provide some basic concepts and results that are either given somewhere, or formulated easily. 

In this section, we fix a $K$-analytic space $X$.

\subsection{Support of a coherent sheaf}

(cf. \cite[Section~2.5]{ducros2018families})

\begin{definition}
	Let $\mathcal{F}$ be a coherent sheaf on $X$, and $\mathrm{Ann}(\mathcal{F})$ the (coherent) annihilator ideal of $\mathcal{F}$ (on the site $X_G$). The \emph{support  of $\mathcal{F}$} is the closed analytic subspace of $X$ defined by $\mathrm{Ann}(\mathcal{F})$, denoted by $\Supp({\mathcal{F}})$. 
\end{definition}
\begin{remark}
Recall the annihilator $\mathrm{Ann}(\mathcal{F})$ of $\mathcal{F}$ is defined as follows: for any analytic domain $V$,
		\[\mathrm{Ann}(\mathcal{F})(V)\coloneqq\{a\in \OO_X(V)\mid a\cdot\mathcal{F}(V)=0\},\]
		which is a coherent ideal. In particular, for any analytic domain $V$, we have $\mathrm{Ann}(\mathcal{F})|_V = \mathrm{Ann}(\mathcal{F}|_V)$. 
\end{remark}
\begin{remark}
If $X=\mathcal{M}(A)$ is affinoid and $\mathcal{F}=\widetilde{M}$, the coherent sheaf associated to some finitely generated $A$-module $M$, then it is easy to see that
		\[\mathrm{Ann}(\mathcal{F}) = \widetilde{\mathrm{Ann}(M)}.\]
\end{remark}

From the definition, we can easy deduce the following lemma.

\begin{lemma}	\label{lemma:coherentsheafonsupport}
	Let $\mathcal{F}$ be a coherent sheaf on $X$, and $Z=\Supp(\mathcal{F})$. Then there is a unique coherent sheaf $\mathcal{G}$ on $Z$ such that $\mathcal{F}=i_*\mathcal{G}$, where $i\colon Z\hookrightarrow X$ is the canonical immersion.
\end{lemma}
\begin{proof}
	By uniqueness, we can glue coherent sheaf $\mathcal{G}$ from local parts, so we can assume that  $X=\mathcal{M}(A)$. It is not hard to see the lemma in this case.
\end{proof}

\subsection{Zariski image of a morphism}

\

As in the theory of schemes, we can define the Zariski image of a morphism of analytic spaces, which has a natural structure of analytic space. We follow the idea in \cite[Subsection~29.6]{stacks-project}.

\begin{lemma}	\label{lemma:largestcoherentsubmodule}
	Let $\mathcal{F}$ be a coherent sheaf on $X$, and $\mathcal{G}\subset \mathcal{F}$ an $\OO_X$-submodule. Then there is a unique coherent $\OO_X$-submodule $\mathcal{G}'\subset\mathcal{G}$ with the following property: for any coherent $\OO_X$-module $\mathcal{H}$, the canonical map
	\[\Hom_{\OO_X}(\mathcal{H},\mathcal{G}')\rightarrow \Hom_{\OO_X}(\mathcal{H},\mathcal{G})\]
	is bijective. In particular, $\mathcal{G}'$ is the largest coherent sheaf contained in $\mathcal{G}$.
\end{lemma}
\begin{proof}
	Let $\{\mathcal{G}_i\}_{i\in I}$ be the set of coherent sheaves contained in $\mathcal{G}$.
	We consider the morphism of $\OO_X$-modules
	\[\varphi\colon\bigoplus\limits_{i\in I}\mathcal{G}_i\to \mathcal{F}.\]
	We claim its image $\mathcal{G}'\subset \mathcal{G}$ is coherent. Let ${^p\mathcal{G}'}\subset \mathcal{G}$ be the image of $\varphi$ as presheaves. Then $\mathcal{G}'$ is the sheafification of ${^p\mathcal{G}'}$, and for any affinoid domain $V=\mathcal{M}(A)$, ${^p\mathcal{G}'}(V)=\sum\limits_i\mathcal{G}_i(V)\subset \mathcal{F}(V)$ is a finitely generated $A$-module. For any analytic domain $W\subset V$, we have ${^p\mathcal{G}'}(W)= {^p\mathcal{G}'}(V)\otimes_{\OO_X(V)}\OO_X(W)$, so $\mathcal{G}'(V) = {^p\mathcal{G}'}(V)$ and $\mathcal{G}'$ is coherent. It is the largest coherent sheaf contained in $\mathcal{G}$.
	
	The map 
	\[\Hom_{\OO_X}(\mathcal{H},\mathcal{G}')\rightarrow \Hom_{\OO_X}(\mathcal{H},\mathcal{G})\]
	is obviously injective. For any homomorphism $\psi\colon \mathcal{H}\to\mathcal{G}\subset\mathcal{F}$, the image $\Im(\psi)\subset \mathcal{G}$ is a coherent sheaf, so $\Im(\psi)\subset \mathcal{G}'$, so $f$ factor through $\mathcal{G}'$. This implies that $\mathcal{G}'$ is the one we want.
	
	The uniqueness is from the Yoneda lemma. 
\end{proof}

\begin{proposition}
	Let $f\colon Y\to X$ be a morphism of $K$-analytic spaces. Then there is a closed analytic subspace $Z$ of $X$ such that
	\begin{enumerate}
		\item [(a)] the morphism $f$ factors through $Z$;
		\item[(b)](Universal property) if $f$ factors through a closed analytic subspace $Z'$ of $X$, then $Z'$ contains $Z$ as a closed analytic subspace.
	\end{enumerate}
	The closed analytic space $Z$ of $X$ is called the \emph{Zariski image} of $f$, denoted by $\Im_\zar(f)$.
\end{proposition}
\begin{proof}
	By (b), if $Z$ exists, then it is unique. It remains to show the existence. Let $\mathcal{I}\coloneqq\Ker(\OO_X\rightarrow f_*\OO_Y)$. By \cref{lemma:largestcoherentsubmodule}, we take the largest coherent $\OO_X$-submodule $\mathcal{J}\subset \mathcal{I}$ and set $Z=V(\mathcal{J})$. It remains to check (a) and (b).
	
	(a) 
	For any affinoid domain $V=\mathcal{M}(A)\subset X$ and any affinoid domain $U=\mathcal{M}(B)\subset f^{-1}(V)$, we have $\mathcal{J}(V)\subset \mathcal{I}(V)\subset \Ker(A\rightarrow B)$, so $U\rightarrow V$ factors through $\mathcal{M}(A/\mathcal{J}(V)) = Z\cap V$ and $U\rightarrow X$ factors through $Z$. Hence $f$ factors through $Z$. 
	
	
	(b) If $f$ factors through a closed subspace $Z'$ of $X$ with $Z'=V(\mathcal{J}')$, then $\mathcal{J}'\subset \mathcal{I}$. By the choice of $\mathcal{J}$, we have $\mathcal{J}'\subset\mathcal{J}$, so $Z'\subset Z$. 
\end{proof}
\begin{remark}
If $X, Y$ are affinoid, $f\colon \mathcal{M}(B)\to \mathcal{M}(A)$ is given by $\varphi\colon A\to B$, then $\Im_\zar(f)=\mathcal{M}(A/\Ker(\varphi))$. However, the Zariski image is not compatible with the G-topology.
\end{remark}

We may expect the Zariski image is exactly the usual image as sets. It is almost true if $Y$ is reduced or $f$ is quasi-compact.

\begin{lemma}\label{lemma:zariskiimageforreducedspace}
Let $f\colon Y\to X$ be a morphism of $K$-analytic space. If $Y$ is reduced, then $\Im_\zar(f)=\overline{f(Y)}^{X_\zar}$ with the reduce closed subspace structure.
\end{lemma}
\begin{proof}
	As a map, $f$ factor through $\overline{f(Y)}^{X_\zar}$. Since $Y$ is reduced, so $f$ factors through $\overline{f(Y)}^{X_\zar}$ with the reduced structure, see \cite[PROPOSITION~4.2~(iii)]{ducros2009les}. It remains to show the universal property of $Y\rightarrow \overline{f(Y)}^{X_\zar}$. If $f$ factors through a closed subspace $Z$ of $X$, then $\overline{f(Y)}^{X_\zar}\subset Z$ as a subset. The containment is also a morphism of analytic spaces since $\overline{f(Y)}^{X_\zar}$ is endowed with the reduced structure.
\end{proof}

\begin{lemma}\label{lemma:zariskiimageforquasicompactmorphism}
	Let $f\colon Y\to X$ be a morphism of $K$-analytic spaces. Assume that $f$ is proper. Then the following statements hold.
	\begin{enumerate}
		\item [(1)] $\mathcal{I}=\Ker(\OO_X\rightarrow f_*\OO_Y)$ is coherent. In particular, $\Im_\zar(f)=V(\mathcal{I})$.
		\item[(2)] $f(Y)= \Im_\zar(f)$.
		\item[(3)] For any analytic domain $V\subset X$, the subspace $\Im_\zar(f)\cap V$ is the Zariski image of $f|_{f^{-1}(V)}\colon f^{-1}(V)\to V$.
	\end{enumerate}
\end{lemma}
\begin{proof}
	(1) 
	Since $f$ is proper, by Kiehl's theorem, see \cite[Proposition~3.3.5]{berkovich1990spectral} or \cref{theorem:Kiehl}, we have that $f_*\OO_Y$ is coherent. Then
	$\mathcal{I}=\Ker(\OO_X\rightarrow f_*\OO_{Y})$ which is coherent. Notice that $\Im_\zar(f)$ is defined by the coherent $\OO_X$-submodule of $\mathcal{I}$, so (1) holds.
	
	(3) This is from (1).
	
	(2) By (3), it suffices to assume that $X=\mathcal{M}(A)$ is affinoid. We take a G-covering $Y=\bigcup\limits_{i=1}^nV_i$ by affinoid domains, and set $Y'=\coprod\limits_{i=1}^nV_i$, $\pi\colon Y'\to Y$ the canonical morphism which is surjective, $f'\coloneq f\pi$. Notice that $f(Y)=f'(Y')$, $\Im_\zar(f)=\Im_\zar(f')$ (by definition, $\Im_\zar(f')\subset\Im_\zar(f)$, conversely, $f$ factors through $\Im_\zar(f)$ since the morphism $Y'\to \Im_\zar(f')$ factors through $Y$). Write $Y'=\mathcal{M}(B)$, and $f'$ is induced by $\varphi\colon A\to B$. We have that $\Im_\zar(f')=\mathcal{M}(A/\Ker(\varphi))$. So the morphism $Y'\rightarrow\Im_\zar(f')$ is induced by an injective homomorphism $A/\Ker(\varphi)\rightarrow B$, hence it is dominant, and $\Im_\zar(f')=\overline{f'(Y')}^{X_\zar}=f(Y)$. This proves (2). 
\end{proof}

\begin{definition} \label{def:graph of a morphism}
	Let $f\colon Y\to X$ be a morphism of $K$-analytic space. The Zariski image of $(\id_Y,f)\colon Y\to Y\times X$ is called the \emph{graph of $f$}, denoted by $\Gamma_f$.
\end{definition}	
\subsection{Codimension}

\

We recall the definition of codimension in 	\cite[1.5.15]{ducros2018families}.
	\begin{definition}
		Let $Y$ be a Zariski-closed subset of $X$. The \emph{codimension $\codim(Y,X)$} of $Y$ in $X$ is defined as follows.
		\begin{itemize}
			\item If both $Y$ and $X$ are irreducible, $\codim(Y,X)\coloneqq\dim_KX-\dim_KY$.
			\item If $Y$ is irreducible, $\codim(Y, X)\coloneqq\sup\limits_{\substack{Z \in \mathrm{Irr}(X)\\ Y\subset Z}}\codim(Y,Z)$.
			\item In the general case, $\codim(Y, X)\coloneqq\inf\limits_{Z \in \mathrm{Irr}(Y)}\codim(Z,X)$.
		\end{itemize}
		For $x\in X$, we define the \emph{codimension of $Y$ in $X$ at $x$} as 
		\[\codim_x(Y,X)\coloneqq\begin{cases}
			\inf\limits_{\substack{Z \in \mathrm{Irr}(Y)\\ x\in Z}}\codim(Z,X) & \text{ if $x\in Y$;}\\
			+\infty & \text{ if $x\not\in Y$.}
		\end{cases}\]
	\end{definition}
	\begin{remark}
			Let $W\subset Z\subset Y\subset X$ be irreducible closed analytic subspaces. Then
			\[\codim(W,Y)=\codim(W,Z)+\codim(Z,Y),\]
			\[\dim_K(Z)+\codim(Z,Y)=\dim_K(Y).\]
	\end{remark}

	\begin{example}[\cite{ducros2007variation}~Proposition~1.11]
		Assume that $X=\mathcal{M}(A)$ is a $K$-affinoid space,  $Y=V(I)$ for some ideal $I\subset A$, and $x\in X$ with image $\xi\in\Spec(A)$. Then
		\begin{enumerate}
			\item [(1)] $\codim(Y,X) = \codim(\Spec(A/I), \Spec(A))$.
			\item[(2)] $\codim_x(Y,X) = \codim_\xi(\Spec(A/I), \Spec(A))$.
		\end{enumerate}
	\end{example}
	\begin{remark}
In particular, (1) implies that \[\codim(\Spec(A_L/I_L), \Spec(A_L)) =\codim(\Spec(A/I), \Spec(A))\] for any complete field extension $L/K$. Or we can write 
			\[\dim_KX-\dim_KY = \codim_{\mathrm{Krull}}(Y,X).\]
	\end{remark}
	

\subsection{A key lemma}


\begin{definition}
	Let $\mathcal{A}$ be a family of objects (e.g. morphisms, irreducible analytic spaces) that we care about. Let $Q$ be a property of objects in $\mathcal{A}$, so we have a disjoint union 
	\[\mathcal{A}=\{S\in \mathcal{A}\mid \text{$S$ satisfies $Q$}\}\cup\{S\in \mathcal{A}\mid \text{$S$ doesn't satisfy $Q$}\}=\colon \mathcal{A}_1\cup \mathcal{A}_2.\]
	Furthermore, let $P$ be a property of objects in $\mathcal{A}_1$. Then for an object $S\in \mathcal{A}$, we say
	\begin{itemize}
		\item $P$ is \emph{well-defined} on $S$ if $S\in \mathcal{A}_1$;
		\item $S$ \emph{satisfies $P$} if $P\in \mathcal{A}_1$ and $S$ satisfies $P$ in the usual sense.
		\item $S$ \emph{doesn't satisfy $P$} if $P\in \mathcal{A}_1$ and $S$ doesn't satisfy $P$ in the usual sense.
	\end{itemize}
\end{definition}
\begin{remark}
In this paper, the property $Q$ is implied when we talk about the property $P$, so we will omit $Q$.
\end{remark}

The following generalized result from \cite[Proposition~4.12]{ducros2009les} is crucial for extending a local result on irreducible closed subsets to be global. 

\begin{lemma}	\label{key lemma for a local property to be global on irreducible subsets}
	Let $P$ be a property of irreducible components of affinoid domains of $X$ satisfying the following properties:
	\begin{enumerate}
		\item[(a)] there is a $G$-covering $X=\bigcup\limits_{i\in I}V_i$ by affinoid domains, the property $P$ is well-defined on each irreducible component of $V_i$ (or simply say that $P$ is well-defined on $V_i$);
		\item[(b)] if $P$ is well-defined on an irreducible component $Z$ of an affinoid domain $V$, then $P$ is well-defined on each irreducible component of $W$ for any affinoid domain $W\subset V$. Moreover, in this case, for any irreducible component $T$ of $W\cap Z$,  we have $T$ satisfies $P$ if and only if $Z$ satisfies $P$.
	\end{enumerate}
	Then there exist Zariski-closed subsets $X^+_P, X_P^-$ of $X$ which are characterized by the following properties: for any affinoid domain $V$ on which $P$ is well-defined, we have
	\[X_P^+\cap V = \bigcup\limits_{\substack{T\in \mathrm{Irr}(V),\\ T \text{ satisfies $P$}}}T,\] \[X_P^-\cap V = \bigcup\limits_{\substack{T\in \mathrm{Irr}(V),\\ \text{$T$ doesn't satisfy $P$}}}T.\]
	Notice that $X=X_P^+\cup X_P^-$.
\end{lemma}
\begin{proof}
	For any affinoid domain $V$ on which $P$ is well-defined, set
	\begin{align*}
		\mathcal{C}^+(V)&\coloneqq\{T\in \mathrm{Irr}(V)\mid T \text{ satisfies $P$}\},\\
		\mathcal{C}^-(V)&\coloneqq\{T\in \mathrm{Irr}(V)\mid T \text{ doesn't satisfy $P$}\},\\
		\mathcal{E}^+(V)&\coloneqq\bigcup\limits_{T\in \mathcal{C}^+(V)}T,\\
		\mathcal{E}^-(V)&\coloneqq\bigcup\limits_{T\in \mathcal{C}^-(V)}T.\\
	\end{align*}
	Let $V$ be an affinoid domain on which $P$ is well-defined, and $W \subset V$ an affinoid domain. Let $Z$ be an irreducible component of $V$ and $T$ an irreducible component of $W$ contained in $Z$. By our assumption, $T\in \mathcal{C}^+(W)$$\Longleftrightarrow$ $Z\in \mathcal{C}^+(V)$. By \cite[COROLLAIRE~4.11]{ducros2009les}, we have $\mathcal{E}^+(W)=\mathcal{E}^+(V)\cap W$ and $\mathcal{E}^-(W)=\mathcal{E}^-(V)\cap W$.
	
	Let $X_P^+$ (resp. $X_P^-$) be the union of $\mathcal{E}^+(V)$ (resp. $\mathcal{E}^-(V)$) where $V$ is  an affinoid domain on which $P$ is well-defined. Then for any affinoid domain $V$ of $X$ on which $P$ is well-defined, we have  $X_P^+\cap V = \mathcal{E}^+(V)$ and $X_P^-\cap V = \mathcal{E}^-(V)$. Since $P$ is well-defined on $V_i$ for some G-covering $X=\bigcup\limits_{i\in I}V_i$ by affinoid domain, and $\mathcal{E}^+(V_i), \mathcal{E}^-(V_i) \subset V_i$ are Zariski-closed, so $X_P^+, X_P^- \subset X$ are Zariski-closed.
\end{proof}
\begin{remark}
In our application, (b) is always from the fact that for an affinoid domain $\mathcal{M}(B)\subset \mathcal{M}(A)$, the homomorphism $A\to B$ is flat.
\end{remark}

\section{Meromorphic functions and Cartier divisors}

\label{Meromorphic functions and Cartier divisors}

The sheaf of meromorphic functions and Cartier divisors are defined on a ringed space in \cite[Section~20, Section~21]{grothendieck1967egaiv4}. However, the definition is not correct, thanks to referee for pointing this out to me, see \cite{kleiman1979misconceptions}. It doesn't work since the restriction of a regular element is not necessarily regular. Fortunately, this can be remedied on analytic spaces (cf. \cite[Section~2]{gubler1998local}). In this section and next section, we will following the idea in \cite[Section~20, Section~21]{grothendieck1967egaiv4} to discuss meromorphic functions, Cartier divisors and cycles. 

Throughout this section, we fix a $K$-analytic space $X$.

\subsection{Meromorphic functions}

\

A reference for meromorphic functions is \cite[\S~2.23]{ducros2021devisser}. For convenience of readers and further uses in our paper, we study meromorphic functions here.


\begin{definition}	\label{sheaf of meromorphic functions}
	For any affinoid domain $V=\mathcal{M}(A) \subset X$, we set $K'_X(V)\coloneqq \Frac(A)$, this will define a presheaf on affinoid domains on $X$. The associated sheaf $K_X$ with respect to the G-topology on $X$ is called the \emph{sheaf of meromorphic functions} on $X$. An element of $K_X(X)$ is called a \emph{meromorphic function} on $X$. The subsheaf of invertible elements of $K_X$ is denoted by $K_X^*$.
\end{definition}
\begin{remark}
For affinoid domains $U=\mathcal{M}(B)\subset V=\mathcal{M}(A)$ of $X$, and $f\in R(A)$, the restriction of $f$ on $U$ is in $R(B)$, this implies that $K_X$ is well-defined. Indeed, this is from the fact $A\rightarrow B$ is flat.
\end{remark}
\begin{remark}
For any analytic domain $V\subset X$, we have
\[K_X(V) = \left\{(s_i)_{i}\in \prod\limits_iK_X'(V_i)\,\middle\vert\, \parbox[c]{.5\linewidth}{ $V=\bigcup\limits_iV_i$ is a G-covering of $V$ with $V_i$ affinoid and $s_i|_{V_{ijk}}= s_j|_{V_{ijk}}$ for some G-covering \ \ \ \ $V_i\cap V_j=\bigcup\limits_kV_{ijk}$ with $V_{ijk}$ affinoid}\right\}\bigg/\sim,\]
where $(s_i)_i \sim (s_j')_{j}$ if for any $i,j$, there exists a G-covering $V_i\cap V_j'=\bigcup\limits_kV_{ijk}$ with $V_{ijk}$ affinoid such that $s_i|_{V_{ijk}} = s'_j|_{V_{ijk}}$.

If $X$ is separated, then it can be simplified as
\[K_X(V) = \left\{(s_i)_{i}\in \prod\limits_iK_X'(V_i)\,\middle\vert\, \parbox[c]{.5\linewidth}{ $V=\bigcup\limits_iV_i$ is an G-covering of $V$ with $V_i$ affinoid and $s_i|_{V_{i}\cap V_j}= s_j|_{V_{i}\cap V_j}$}\right\}\bigg/\sim,\]
where $(s_i)_i \sim (s_j')_{j}$ if for any $i,j$, $s_i|_{V_{i}\cap V'_j} = s'_j|_{V_{i}\cap V'_j}$.
\end{remark}
 \begin{remark}
 	For any affinoid domain $V\subset X$, the canonical map $K_X'(V)\rightarrow K_X(V)$ is injective. In particular, $\OO_X\subset K_X$.
 	\begin{proof}
 		Given an affinoid domain $V$ and any finite G-covering $V=\bigcup\limits_{i=1}^nV_i$ by affinoid domains, let $A=\OO_X(V)$ and $A_i=\OO_X(V_i)$. We consider the restriction map $\Frac(A)\rightarrow \prod\limits_{i=1}^n \Frac(A_i)$.
 		Let $a/b\in \Frac(A)$ be such that its restriction on $\Frac(A_i)$ is $0$ for any $i$, i.e. $a=0\in A_i$. This implies that $a=0\in A$ by Tate's acyclic theorem. Hence $K_X'(V)\hookrightarrow K_X(V)$.
 		We take a G-covering $X=\bigcup\limits_{i\in I}V_i$ by affinoid domains. Then the injective map $\OO_X(V_i)\hookrightarrow K_X'(V_i)$ will induce $\OO_X\hookrightarrow K_X$. 
 	\end{proof}
 \end{remark}

\begin{definition}	\label{meromorphic section}
	For an $\OO_X$-module $\mathcal{F}$, we call $\mathcal{F}\otimes_{\OO_X}K_X$ the \emph{sheaf of meromorphic sections of $\mathcal{F}$}, and we have a canonical map
	\[\id_{\mathcal{F}}\otimes i\colon \mathcal{F}\to \mathcal{F}\otimes_{\OO_X}K_X.\]
	The sheaf $\mathcal{F}$ is called \emph{strictly without torsion} if $\id_{\mathcal{F}}\otimes i$ is injective.
	
	A global section of $\mathcal{F}\otimes_{\OO_X}K_X$ is called a \emph{meromorphic section} of $\mathcal{F}$ on $X$. 
	
	If $\mathcal{F}$ is coherent on $X$, we say a meromorphic section $s$ on $X$ is defined on a Zariski-open subset $V$ if $s|_V$ is in the image of $\mathcal{F}(V)$ via $\id_{\mathcal{F}}\otimes i$. If moreover, $\mathcal{F}$ is strictly without torsion, then there is a maximal Zariski-open subset $V$ on which s is defined, such $V$ is called the \emph{domain of definition} of $s$, denoted by $\mathrm{dom}(s)$ (i.e. $s\in \mathcal{F}(\mathrm{dom}(s))$).
\end{definition}
\begin{remark}
Notice that $\mathcal{F}\rightarrow \mathcal{F}\otimes_{\OO_X}K_X$ is the sheafification of the presheaf given by 
\[V\mapsto \mathcal{F}(V)\otimes_{\OO_X(V)}K_X'(V)\]
for any affinoid domain $V$. So for any analytic domain $V\subset X$, we have $(\mathcal{F}\otimes_{\OO_X}K_X)|_V\simeq \mathcal{F}|_V\otimes_{\OO_V}K_V$. In particular, $K_X|_V = K_V$.
\end{remark}
\begin{remark}
A locally free $\OO_{X_G}$-module $\mathcal{F}$ is strictly without torsion. Moreover, $\mathcal{F}\otimes_{\OO_X}K_X$ is a locally free $K_X$-module, here, we view $(X_G, K_X)$ as a G-ringed space.
\end{remark}

For a good $K$-analytic space, the sheaf of meromorphic functions can be given in a similar way in \cite[Section~20]{grothendieck1967egaiv4}, and will have some good properties, i.e. properties for schemes can be extended to good analytic spaces.

If  $X$ is good, and $x \in X$ is rigid, we have that \[\OO_{X,x}=\lim\limits_{\overset{\longrightarrow}{V}}\OO_X(V)\] 
where $V$ runs through affinoid domains containing $x$, see \cite[Section~2.3]{berkovich1990spectral}. In particular, it suffices that $V$ runs through (strictly) affinoid neighborhoods of $x$ in $X$.

\begin{proposition}	\label{basic properties of meromorphic functions}
	Assume that $X$ is good. For any analytic domain $V\subset X$, set
	\[\mathcal{R}(V)\coloneqq\{s\in \OO_X(V)\mid s_x\in R(\OO_{X,x}) \text{ for any $x\in V$}\}\subset \OO_X(V),\]
	which defines a sheaf on $X$. Then the following statements hold:
	\begin{enumerate}
		\item [(1)] For any affinoid domain $V\subset X$, we have $\mathcal{R}(V)=R(\OO_X(V))$. In particular, $K_X$ is the sheafification of the following presheaf: for any analytic domain $V\subset X$,
		\[V\mapsto \mathcal{R}(V)^{-1}\OO_X(V).\]
		\item[(2)] For any rigid point $x\in X$, we have $K_{X,x}' \simeq \Frac(\OO_{X,x})$. For any analytic domain $V\subset X$, the canonical homomorphism $K'_X(V)\hookrightarrow \prod\limits_{x\in V \text{ rigid }}K'_{X,x}$ is injective.
	\end{enumerate} 
\end{proposition}
\begin{proof}
	Notice that the presheaf $\mathcal{R}$ is a sheaf. Since $\mathcal{R}$ is a subpresheaf of $\OO_X$, and if $V=\bigcup\limits_{i\in I}V_i$ is a G-covering of an analytic domain $V$, $a_i\in \mathcal{R}(V_i)$ such that $a_i|_{V_i\cap V_j} = a_j|_{V_i\cap V_j}$ then there exists $a\in \OO_X(V)$ such that $a|_{V_i} = a_i$, then $a\in \mathcal{R}(V)$.
	
	(1) For any affinoid domain $V\subset X$ and $a\in \OO_X(V)$, we have $a$ is regular $\Longleftrightarrow$ $a\in \OO_{X,x}$ regular for any $x\in V$. Indeed, "$\Longrightarrow$" is from the flatness, for "$\Longleftarrow$", if $a\in \OO_{X,x}$ is regular, then there is an affinoid neighborhood $V_x$ of $x$ in $V$ such that $a\in R(\OO_X(V_x))$ (since $\Ker(\OO_X(V)\overset{\cdot a}{\rightarrow} \OO_X(V))$ is finitely generated). Then $a\in R(\OO_X(V))$ since $V=\bigcup\limits_{x\in V}V_x$ is a G-covering. So $\mathcal{R}(V) = R(\OO_X(V))$. Hence $K_X'(V)=\Frac(\OO_X(V))$.
	
	(2) By definition, we have a map
	\[\lim\limits_{\substack{\longrightarrow\\V}}K_X'(V)\rightarrow \mathcal{R}_x^{-1}\OO_{X,x}\]
	which is surjective, where $V$ runs through affinoid neighborhoods of $x$. If $a/b\in K_X'(V)$ with $V$ affinoid neighborhood of $x$ such that $a/b=0\in \mathcal{R}_x^{-1}\OO_{X,x}$, i.e. there is $c\in \mathcal{R}_{x}$ such that $ac=0$. We can assume that $c\in \OO_X(V)$, then $a/b=0\in K_X'(V)$. 
	
	It remains to show that $\mathcal{R}_x=R(\OO_{X,x})$. We have an injective map  $\mathcal{R}_x\hookrightarrow R(\OO_{X,x})$ by definition. Conversely, for $a\in R(\OO_{X,x})$, we consider an affinoid neighborhood $V$ of $x$ with $A=\OO_X(V)$ such that $a\in A$, then 
	\[\xymatrix{0\ar[r]& \mathrm{Ann}(a)\ar[r]& A\ar[r] &A}.\]
	Since $\mathrm{Ann}(a)$ is finitely generated and $a\in R(\OO_{X,x})$, so we can find an affinoid neighborhood $U\subset V$ of $x$ with $B=\OO_X(U)$  such that $\mathrm{Ann}(a)\otimes_AB = 0$. So $a\in R(B)$. By (1), we know that $\mathcal{R}_x=R(\OO_{X,x})$.
	
	If $a/b\in K_X'(V)$ such that $0=a/b\in K_{X,x}'$ for any rigid $x\in V$, then there exists an affinoid neighborhood $V_x$ of $x$ such that $0=a/b\in K_X'(V_x)$. Since $\mathcal{R}(V_x) = R(\OO_X(V_x))$, we have $0=a\in \OO_X(V_x)$ and $a=0\in K_X'(V)$, $a/b=0$.
\end{proof}


\subsection{Cartier divisors}

\begin{definition} 	\label{cartier divisors}
	We denote the group ${H}^0(X_G, K_X^*/\OO_X^*)$ by $\Div(X)$. The elements of $\Div(X)$ are called \emph{Cartier divisors} of $X_G$.
	
	Let $f\in H^0(X_G, K_X^*)$, its image in $\Div(X)$ is called a \emph{principal Cartier divisor} and denoted by $\mathrm{div}(f)$.
	
	We say that two Cartier divisor $D_1, D_2$ are \emph{linearly equivalent} if $D_1-D_2$ is principal, write $D_1\sim D_2$. We denote $\CaCl(X)$ the group of equivalent class of Cartier divisors.
	
	A Cartier divisor $D$ is called \emph{effective} if it is in the image of the canonical map $H^0(X_G,(\OO_X\cap K_X^*)/\OO_X^*)\rightarrow H^0(X_G, K_X^*/\OO_X^*)$, write $D\geq 0$. The set of effective Cartier divisors is denoted by $\Div_+(X)$.
\end{definition}
\begin{remark}
The exact sequence of sheaves
		\[\xymatrix{0\ar[r] & \OO_X^*\ar[r] & K_X^* \ar[r]& K_X^*/\OO_X^*\ar[r] &0}\]
		will induce a long exact sequence
		\[\xymatrix{0\ar[r] & H^0(X_G,\OO_X^*) \ar[r]& H^0(X_G,K_X^*) \ar[r]& \Div(X) \ar[r] &\\ &H^1(X_G,\OO_X^*) \ar[r]& H^1(X_G, K_X^*) \ar[r]& \cdots&}.\]
\end{remark}		
\begin{remark}
 We can represent a Cartier divisor $D$ by a system $\{(U_i, f_i)\}_{i\in I}$, where $X=\bigcup\limits_{i\in I}U_i$ is a G-covering by affinoid domains, and $f_i=a_i/b_i\in K_X'(U_i)$ such that $f_i|_{U_i\cap U_j}\in f_j|_{U_i\cap U_j} \OO_X(U_i\cap U_j)^*$ for every $i,j\in I$. Two systems $\{(U_i,f_i)\}_{i\in I}$ and $\{(V_j,g_j)\}_{j\in J}$ represent the same Cartier divisor if only only if $f_i|_{U_i\cap V_j}\in g_j|_{U_i\cap V_j}\OO_X(U_i\cap V_j)^*$ for any $i\in I, j\in J$.
		
		If $D_1=\{(U_i,f_i)\}_{i\in I}$ and $D_2=\{(V_j,g_j)\}_{j\in J}$, then $D_1+D_2 = \{(W_{ijk},f_ig_j)\}_{i\in I,j\in J}$, where $U_i\cap V_j=\bigcup\limits_{k}W_{ijk}$ is a G-covering by affinoid domains.
\end{remark}

\begin{remark}
If $X=\mathcal{M}(A)$ is an affinoid domain, let $\mathcal{X}=\Spec(A)$, then we have an injection $\Div(\mathcal{X})\hookrightarrow \Div(X)$. Indeed, the homomorphism is given by the fact that a Zariski-open subset is a domain, and the injectivity is from the fact that for a given G-covering $X=\bigcup\limits_{i}U_i$ by affinoid domains, $f\in A$ is invertible if and only if $f\in \OO_X(U_i)$ is invertible for any $i$.
\end{remark}

\begin{proposition}\label{prop:divisorandlinebundles}
	\begin{enumerate}
		\item[(1)] For any divisor $D=\{(U_i,f_i)\}_{i\in I}\in \CaCl(X)$, we can associate a subsheaf $\OO_X(D)\subset K_X$ defined by $\OO_X(D)|_{U_i}=f_i^{-1}\OO_X|_{U_i}$, which is an invertible sheaf and independent of the choice of representative. Moreover, $D\geq 0$ $\Longleftrightarrow$ $\OO_X(-D)\subset \OO_X$.
		\item[(2)] The construction above gives a homomorphism of groups $\rho\colon \Div(X)\to \Pic(X), \ \ D\mapsto \OO_X(D)$.
		\item[(3)] The homomorphism $\rho$ induces an injective homomorphism $\CaCl(X)\rightarrow \Pic(X)$ with image $\Im\rho$ corresponding to invertible sheaves contained in $K_X$.
		\item[(4)] If $X$ is affinoid and reduced, then $\rho\colon \CaCl(X)\to \Pic(X)$ is an isomorphism.
	\end{enumerate} 
\end{proposition}
\begin{proof}
	We follow the idea of the proof of \cite[Proposition~7.1.18]{liu2006algebraic}.
	
	(1) Assume $D=\{(V_j,g_j)\}_{j\in J}$ is another representative. Then
	\begin{align*}
		\OO_X(D)|_{U_i\cap V_j}=f_i^{-1}\OO_X|_{U_i\cap V_j}=(g_ju)^{-1}\OO_X|_{U_i\cap V_j} = g_j^{-1}\OO_X|_{U_i\cap V_j}
	\end{align*}
	where $u\in \OO_{X}(U_i\cap V_j)^*$, this implies $\OO_X(D)$ is independent of the choice of representative. By construction, $\OO_X(D)\in \Pic(X)$, and $D\geq0$ if and only if $\OO_X(D)\subset\OO_X$.
	
	(2) The map is a homomorphism. Indeed, let $D_1=\{(f_i, U_i)\}_{i\in I}$ and $D_2=\{(g_i, U_i)\}_{i\in I}$, then
	\[\rho(D_1+D_2)|_{U_i} = f_i^{-1}g_i^{-1}\OO_X|_{U_i}\simeq f_i^{-1}\OO_X|_{U_i}\otimes_{\OO_X|_{U_i}}g_i^{-1}\OO_X|_{U_i},\]
	and this isomorphism is compatible on the intersection $U_i\cap U_j$. 
	
	(3) If $D=\{(U_i,f_i)\}_{i\in I}=\mathrm{div}(f)$ is a principal divisor with $f\in H^0(X_G, K_X^*)$ and $f_i=f|_{U_i}\in K'_X(U_i)$, where $X=\bigcup\limits_{i\in I}U_i$ is a G-covering of $X$ by affinoid domains. Then $f^{-1}\in \OO_X(D)(X)$ because of the following exact sequence
	\[\xymatrix{0\ar[r] & \OO_X(D)(X) \ar[r] & \prod\limits_{i\in I}f_i^{-1}\OO_X(U_i) \ar[r] & \prod\limits_{i\in I}f_i^{-1}\OO_X(U_i\cap U_j)}.\]
	So we can define the morphism $\OO_X\rightarrow \OO_X(D), \ \ a\mapsto af^{-1}$. It is an isomorphism since it is an isomorphism on each $U_i$. Hence we have a homomorphism $\CaCl(X)\rightarrow \Pic(X)$. 
	
	If $D=\{(U_i,f_i)\}_{i\in I} \in \Div(X)$ such that $\OO_X(D)\simeq \OO_X$, then there is $g\in \OO_X(D)(X)$ such that the morphism $\OO_X\overset{\sim}{\rightarrow}\OO_X(D), \ \ a\mapsto ag$ is an isomorphism. Since $\OO_X(D)|_{U_i}\simeq f_i^{-1}\OO_X|_{U_i} = g|_{U_i}\OO_{X}|_{U_i}$
	and $f_i^{-1}\in K_X'^*(U_i)$, $g|_{U_i}=f_i^{-1}u_i\in K_X'^*(U_i)\subset K_X^*(U_i)$ with $u_i\in \OO_X^*(U_i)$, we have $g\in K_X^*(X)$ and $D=\{(U_i,f_i)\}_{i\in I} =\{(U_i, g^{-1}|_{U_i})\}_{i\in I}$ is principal.
	
	By definition, we know that $\OO_X(D) \subset K_X$. Conversely, for $L\in \Pic(X)$ with $L\subset K_X$, there is a G-covering $X=\bigcup\limits_{i\in I}U_i$ by affinoid domains such that $\OO_X|_{U_i}\simeq L|_{U_i}$. We take $g_i\in L(U_i)$ which is mapped to $1$. Then $g_i\in K_X(U_i)$ and $L|_{U_i}=g_i\OO_X|_{U_i}$, moreover, there is $f_i\in K_X^*(U_i)$ such that $f_ig_i=1$ because of the isomorphism. On $U_i\cap U_j$, we have 
	\[L|_{U_i\cap U_j} = f_i^{-1}\OO_X|_{U_i\cap U_j} = f_j^{-1}\OO_X|_{U_i\cap U_j},\]
	so there is $u\in \OO_X^*(U_i\cap U_j)$ such that $f_i^{-1}|_{U_i\cap U_j}=u f^{-1}_j|_{U_i\cap U_j}$. Then $L=\OO_X(D)$, where $D=\{(U_i,f_i)\}_{i\in I} \in \Div(X)$.
	
	(4) 
	Let $\mathcal{X}=\Spec(\OO_X(X))$, then $\CaCl(\mathcal{X})\simeq \Pic(\mathcal{X})$, see \cite[Corollary~1.19]{liu2006algebraic}. It suffices to show the surjectivity of $\Div(X)\to\Pic(X)$. We have a commutative diagram
	\[\xymatrix{\Div(\mathcal{X})\ar@{^{(}->}[r] \ar[d]_{\rotatebox{90}{$\sim$}}&\Div(X)\ar[d]^\rho\\ 
		\Pic(\mathcal{X})\ar[r]^\sim & \Pic(X)},\]
	here the isomorphism $\Pic(\mathcal{X})\simeq \Pic(X)$ is from $\mathcal{C}oh(\mathcal{X})\simeq \mathcal{C}oh(X)$ and Tate's acyclicity theorem, see the proof of \cite[Propostion~1.3.4~(iii)]{berkovich1993etale}. So (4) holds. 
\end{proof}
\begin{remark}
We know that $H^1(X_G,\OO_X^*)\simeq\Pic(X)$, then $\rho$ is the connecting map of the long exact sequence.
\end{remark}

\begin{example}	\label{rational sections}
	Let $L$ be an invertible sheaf on a normal $K$-analytic space $X$. Let $s\in H^0(X, L\otimes_{\OO_X}K_X)$ be a rational section which is non-zero on each irreducible component. Let $X=\bigcup\limits_{i\in I}U_i$ be a G-covering of $X$ by integral affinoid domains such that ${L}|_{U_i}$ is free and generated by an element $e_i$. Then these exist $f_i\in K_X^*(U_i)$ such that $s|_{V_i}=f_ie_i$. Moreover $\mathrm{div}(s)\coloneqq\{(U_i, f_i)\}_{i\in I}$ is a Cartier divisor such that $\OO_X(\mathrm{div}(s))\simeq L$.
\end{example}

\subsection{Inverse image of a Cartier divisor}

\

Next we consider the restriction of Cartier divisors on a closed analytic subspace. 

\begin{definition}
	Let $D\in \Div(X)$, and $Z\in \overline{\mathrm{Irr}(X)}$ with reduced analytic space structure.
	We say \emph{$D$ intersects $Z$ properly} if there is a G-covering
	$X=\bigcup\limits_{i\in I}U_i$ by affinoid domains such that $D=\{(U_i, a_i/b_i)\}_{i\in I}$ with the images $\overline{a}_i, \overline{b}_i\in R(\OO_Z(U_i\cap Z))$.  The set of Cartier divisor intersecting $Z$ properly is a subgroup of $\Div(X)$, denoted by $G_{Z/X}$. 
\end{definition}
\begin{remark}
 There is a natural homomorphism $G_{Z/X}\rightarrow \Div(Z)$ denoted by $D\mapsto D|_Z$, compatible with the homomorphism $\OO_X\rightarrow i_*\OO_Z$. Moreover, we have a canonical isomorphism $\OO_X(D)|_Z\simeq \OO_Z(D|_Z)$.
\end{remark}

\section{Cycles, flat pull-backs and proper push-forwards}

\label{Cycles}

In this section, we fix a $K$-analytic space $X$.

\subsection{Cycles}

\begin{definition}	\label{cycles}
	A \emph{prime cycle} on $X$ is an element in $\overline{\mathrm{Irr}(X)}$. A \emph{cycle} on $X$ is a formal sum $\alpha=\sum\limits_{Z\in \overline{\mathrm{Irr}(X)}}n_Z[Z]$ with $n_Z\in \Z$ which is G-locally finite, i.e. the set 
	\[\{Z\in \overline{\mathrm{Irr}(X)}\mid Z\cap V\not=\emptyset, n_Z\not=0\}\]
	is finite for any affinoid domain $V$. 
	The coefficient $n_Z$ is called the \emph{multiplicity of $\alpha$ at $Z$}, denoted by $\mathrm{mult}_Z(\alpha)$. We say that a cycle $\alpha$ is \emph{positive} if $\mathrm{mult}_Z(\alpha)\geq 0$ for any $Z\in \overline{\mathrm{Irr}(X)}$. The set of cycles (resp. positive cycles) is denoted by $Z(X)$ (resp. $Z_+(X)$).
	
	The union of the $Z$ such that $n_Z\not=0$ is called the \emph{support of $\alpha$}, denoted by $\Supp(\alpha)$. It is a Zariski-closed subset of $X$. In particular, $\Supp(0)=\emptyset$.
	
	A cycle $\alpha$ is (purely) \emph{of codimension $r$} (resp. \emph{of dimension $r$}) if any $Z\in \overline{\mathrm{Irr}(X)}$ with $n_Z\not=0$ has codimension $r$ (resp. dimension $r$). The cycles of codimension $r$ (resp. of dimension $r$) form a subgroup $Z^r(X)$ (resp. $Z_r(X)$) of the group of cycles on $X$.
\end{definition}
\begin{remark}
For a positive cycle $\alpha=\sum\limits_{Z\in \overline{\mathrm{Irr}(X)}}n_Z[Z]$ and any $Z\in \overline{\mathrm{Irr}(X)}$ with $n_Z\geq 1$, we can endow $Z$ with the reduced subscheme structure, then $Z=V(\mathcal{I}_Z)$ is an integral closed analytic subspace of $X$, where $\mathcal{I}_Z$ is the coherent sheaf of ideal defining $Z$. We view $\alpha$ as a closed analytic subspace defined by the sheaf of ideal $\mathcal{I}_\alpha\coloneqq\prod\limits_{Z\in \overline{\mathrm{Irr}(X)}}\mathcal{I}_Z^{n_Z}$ and we have a canonical closed immersion $j\colon \alpha=V(\mathcal{I}_\alpha)\hookrightarrow X$. This induces a homomorphisms of semi-groups
		\[Z_+(X)\rightarrow \{\text{closed analytic subspace of $X$}\}=\{\text{coherent sheaves of ideals on $X$}\}.\] 
\end{remark}

\begin{example}
	Let $X=\mathcal{M}(A)$ be a $K$-affinoid space. Set $\mathcal{X}=\Spec(A)$. Then 
	\[\Div(X)\hookrightarrow \Div(\mathcal{X}),\]
	\[Z^*(X)\simeq Z^*(\mathcal{X}).\]
	The first arrow is also an isomorphism if $X$ is regular, see \cref{prop:cartierdivisorandweildivisor}.
\end{example}  

\begin{lemma}
	Let $\alpha\in Z_+(X)$ with associated sheaf of ideal $\mathcal{I}_\alpha$. Then $V(\mathcal{I}_\alpha)=\Supp(\alpha)$ with $\mathrm{Irr}(V(\mathcal{I}_\alpha))=\{\text{maximal elements in $\alpha$}\}$.
\end{lemma}
\begin{proof}
	This is local, and we can deduce this lemma from the example above.
\end{proof}

\begin{definition}
	Let $V$ be an analytic domain of $X$. For a prime cycle $[Z]\in Z^*(X)$, we set
	\[[Z\cap V]\coloneqq\sum\limits_{Z_i\in  \mathrm{Irr}(Z\cap V)}[Z_i].\]
	This can be extended to general cycles by linearity, i.e. for a cycle $\alpha=\sum\limits_{Z\in \overline{\mathrm{Irr}(X)}}n_Z[Z]$, its \emph{restriction} is defined as $\alpha|_V\coloneq \sum\limits_{Z\in \overline{\mathrm{Irr}(X)}}n_Z[Z\cap V]\in Z^*(V)$.
\end{definition}

\begin{lemma} 	\label{lemma:localequivalence}
	Let $X=\bigcup\limits_{i\in I}V$ be a G-covering of by affinoid domains, and $\alpha, \beta\in \mathrm{Div}(X)$ (resp. $Z^*(X)$). Then $\alpha=\beta$ $\Longleftrightarrow$ $\alpha|_{V_i}=\beta|_{V_i}$ for any $i\in I$.
\end{lemma}
\begin{proof}
	It suffices to show the "if" part. If $\alpha, \beta\in \Div(X)$, then the result holds from the expression of Cartier divisors. If $\alpha=\sum\limits_{Z}n_Z[Z], \beta=\sum\limits_{Z}m_Z[Z]\in Z^k(X)$ such that $\alpha|_{V_i}=\beta|_{V_i}$ for any $i\in I$, then $n_Z[Z\cap V_i] = m_Z[Z\cap V_i]$ for any $Z\in \overline{\mathrm{Irr}(X)}$ with $Z\cap V_i\not=\emptyset$, so $n_Z=m_Z$. 
\end{proof}
\subsection{Cycle associated to a coherent sheaf}

\

We will construct a homomorphism $\Div({X})\rightarrow Z^1(X)$ as we do in algebraic geometry. Recall, for a Noetherian affine scheme $\mathcal{X}=\Spec(A)$, a coherent sheaf $\mathcal{F}=\widetilde{M}$ on $\mathcal{X}$, and an irreducible component $Z$ of $\Supp(\mathcal{F})$, we set $\mathrm{mult}_Z(\mathcal{F}) \coloneqq \mathrm{length}_{A_{\mathfrak{p}}}(M_{\mathfrak{p}})$, called the multiplicity of $Z$ in $\mathcal{F}$, where $\mathfrak{p}\in \mathcal{X}$ is the prime ideal corresponding to $Z$. For a divisor $D\in \Div(\mathcal{X})$ and a codimension one prime cycle  $Z=\overline{\{z\}}\in Z^1(\mathcal{X})$, we set $\mathrm{mult}_Z(D)\coloneqq\mathrm{mult}_{\OO_{\mathcal{X},z}}(D_z)$ the multiplicity of $Z$ in $D$. For an affinoid space $\mathcal{M}(A)$, we have similar notation via $\mathcal{C}oh(\mathcal{M}(A))\simeq \mathcal{C}oh(\Spec(A))$.

\begin{lemma} 	\label{multiplicity of a coherent sheaf}
	Let $\mathcal{F}$ be a coherent sheaf on $X$. For any irreducible component $Z$ of $\Supp(\mathcal{F})$ with reduced analytic space structure, and an affinoid domain $V\subset X$ with $Z\cap V\not=\emptyset$, we set 
	\[\mathrm{mult}_Z(\mathcal{F})\coloneq\mathrm{mult}_{T}(\mathcal{F}|_V)\]
	where $T$ is an irreducible component of $Z\cap V$ with $\overline{T}^{\Supp(\mathcal{F})_{\mathrm{Zar}}}=Z$. Then $\mathrm{mult}_Z(\mathcal{F})$ is a positive integer which is independent of the choice of $T$ and $V$. We call $\mathrm{mult}_Z(\mathcal{F})$ the \emph{multiplicity of $Z$ in $\mathcal{F}$}.
\end{lemma}
\begin{proof}
	For a fixed irreducible component $Z$ of $\Supp(\mathcal{F})$, and any affinoid domain $V, W\subset X$ with $W\subset V$, $Z\cap W\not=\emptyset$, we claim that 
	\[\mathrm{mult}_{T}(\mathcal{F}|_V) = \mathrm{mult}_{T'}(\mathcal{F}|_W)\]
	where $T\in\mathrm{Irr}(Z\cap V)$ (resp.$T'\in\mathrm{Irr}(Z\cap W)$) with $\overline{T'}^{V_{\mathrm{Zar}}}=T$, $\overline{T}^{X_{\mathrm{Zar}}}=Z$. Indeed, let $V=\mathcal{M}(A), W=\mathcal{M}(B)$ and $\mathcal{F}|_V=\widetilde{M}$. We shall show that 
	\[\mathrm{length}_{A_{\mathfrak{p}}}(M_{\mathfrak{p}}) =\mathrm{length}_{B_{\mathfrak{q}}}(M_{\mathfrak{p}}\otimes_{A_{\mathfrak{p}}}B_{\mathfrak{q}})\]
	where $\mathfrak{p}\subset A$ (resp. $\mathfrak{q}\subset B$) is the prime ideal corresponding to $T$ (resp. $T'$). By \cite[\S~2.1.3~(1)]{ducros2018families}, $A_{\mathfrak{p}}\to B_{\mathfrak{q}}$ is regular. Since $\mathfrak{q}\cap A=\mathfrak{p}$ and $\dim A_{\mathfrak{p}} = \dim B_{\mathfrak{q}}=0$, we have that $B_{_{\mathfrak{q}}}\otimes_{A_{\mathfrak{p}}}k(\mathfrak{p})=B_{_{\mathfrak{q}}}/\mathfrak{p}B_{_{\mathfrak{q}}}$ is a field, so $\mathfrak{p}A_{\mathfrak{p}}\otimes_{A_{\mathfrak{p}}}B_{\mathfrak{q}}=\mathfrak{p}B_{\mathfrak{q}}=\mathfrak{q}B_{\mathfrak{q}}$. Hence 
	\begin{align*}
\mathrm{length}_{A_{\mathfrak{p}}}(M_{\mathfrak{p}})& = \sum\limits_{n=1}^\infty\dim_{k(\mathfrak{p})}(M_{\mathfrak{p}}\otimes_{A_{\mathfrak{p}}}\mathfrak{p}^n/\mathfrak{p}^{n+1})\\
& = \sum\limits_{n=1}^\infty\dim_{k(\mathfrak{q})}(M_{\mathfrak{p}}\otimes_{A_{\mathfrak{p}}}\mathfrak{p}^n/\mathfrak{p}^{n+1}\otimes_{k(\mathfrak{p})}k(\mathfrak{q}))\\
&= \sum\limits_{n=1}^\infty\dim_{k(\mathfrak{q})}(M_{\mathfrak{p}}\otimes_{A_{\mathfrak{p}}}\mathfrak{p}^nA_{\mathfrak{p}}\otimes_{A_{\mathfrak{p}}}k(\mathfrak{q}))\\
&= \sum\limits_{n=1}^\infty\dim_{k(\mathfrak{q})}(M_{\mathfrak{p}}\otimes_{A_{\mathfrak{p}}}\mathfrak{q}^n/\mathfrak{q}^{n+1})\\
&=\mathrm{length}_{B_{\mathfrak{q}}}(M_{\mathfrak{p}}\otimes_{A_{\mathfrak{p}}}B_{\mathfrak{q}}).
	\end{align*}

	To show the lemma, we apply \cref{key lemma for a local property to be global on irreducible subsets}. Let $Z\in Z^1(X)$ be a prime cycle, and $m=\mathrm{mult}_{T}(\mathcal{F}|_V)$ for some affinoid domain $V\subset X$ with $Z\cap V\not=\emptyset$, where $T\in\mathrm{Irr}(Z\cap V)$ with $\overline{T}^{X_{\mathrm{Zar}}}=Z$. For $V$ given as before, we say an irreducible component $T\in\mathrm{Irr}(Z\cap V)$ satisfies $P$ if $\mathrm{mult}_{T}(\mathcal{F}|_V)=m$. After replacing $X$ by $Z$, from our claim, we see that $P$ satisfies the hypothesis in Lemma~\ref{key lemma for a local property to be global on irreducible subsets}. Then there are Zariski-closed subsets $Z_P^+, Z_P^-$ of $Z$ such that 
	\[Z_P^+\cap V = \bigcup\limits_{\substack{T\in \mathrm{Irr}(Z\cap V),\\ T \text{ satisfies $P$}}}T,\] \[Z_P^-\cap V = \bigcup\limits_{\substack{T\in \mathrm{Irr}(Z\cap V),\\ \text{$T$ doesn't satisfy $P$}}}T,\]
	and $Z=Z_P^+\cup Z_P^-$. Since $Z$ is irreducible and there is some $T \subset Z_P^+$, we have $Z=Z_P^+$. This implies the lemma.
\end{proof}

\begin{definition}	\label{cycle associated to a coherent sheaf}
	For a coherent sheaf $\mathcal{F}$ on $X$  with $\codim(\Supp(\mathcal{F}),X)\geq k$, we set \[[\mathcal{F}]^k\coloneq\sum\limits_{Z\in \mathrm{Irr}(\Supp(\mathcal{F}))^k}\mathrm{mult}_Z(\mathcal{F})[Z]\in Z^k(X),\]
	called the \emph{cycle associated to $\mathcal{F}$ with codimension $k$}.
\end{definition}
\begin{remark}\label{remark:cycle associated to a coherent sheaf}
By \cref{multiplicity of a coherent sheaf}, it is not hard to have the following result. Let $V=\mathcal{M}(A)\subset X$ be an affinoid domain, and $\mathcal{F}$ a coherent sheaf on $X$. Set $\mathcal{V}=\Spec(A)$ and $\mathcal{F}^{\mathrm{al}}_V$ the coherent sheaf on $\mathcal{V}$ corresponding to $\mathcal{F}|_V$. Then 
		\[[\mathcal{F}|_V]^k = [\mathcal{F}^{\mathrm{al}}_V]^k,\]
		here we identify $Z^*(V)\simeq Z^*(\mathcal{V})$. 
\end{remark}

\begin{definition} 	\label{cycle associated to closed analytic subspace}
	For a closed analytic subspace $Y$ of $X$ with $\codim(Y,X)\geq k$, we set
	\[\mathrm{mult}_Z(Y)\coloneqq \mathrm{mult}_Z(\OO_Y),\]
	for any $Z\in \mathrm{Irr}(Y)$, called the \emph{multiplicity of $Z$ in $Y$}, and set 
	\[[Y]^k\coloneqq\sum\limits_{\substack{Z\in \mathrm{Irr}(Y)\\ Z\in Z^k(X)}}\mathrm{mult}_Z(Y)[Z]\in Z^k(X),\]
	called the \emph{cycle associated to $Y$ with codimension $k$}.
\end{definition}

\subsection{Weil divisors}

\begin{definition} 	\label{weil divisor}
	An element in $Z^1(X)$ is called a \emph{Weil divisor} on $X$.
\end{definition}

\begin{lemma} 	\label{multiplicity of a divisor}
	Let $D\in \Div(X)$. For any prime cycle $Z\in Z^1(X)$, and any affinoid domain $V\subset X$ with $Z\cap V\not=\emptyset$, $D|_{V}\in K'_X(V)$, we set 
	\[\mathrm{mult}_Z(D)\coloneqq\mathrm{mult}_{T}(D|_V)\]
	where $T\in\mathrm{Irr}(Z\cap V)$ with $\overline{T}^{X_{\mathrm{Zar}}}=Z$. Then $\mathrm{mult}_Z(D)$ is independent of the choice of $T$ and $V$. We call $\mathrm{mult}_Z(D)$ the \emph{multiplicity of $Z$ for $D$}.
\end{lemma}
\begin{proof}
	The proof is similar with the one of Lemma~\ref{multiplicity of a coherent sheaf}.
	
	For any prime cycle $Z\in Z^1(X)$ and any affinoid domain $V, W\subset X$ with $W\subset V$, $Z\cap W\not=\emptyset$, $D|_{V}\in K'_X(V)$, we claim that 
	\[\mathrm{mult}_{T}(D|_V) = \mathrm{mult}_{T'}(D|_W),\]
	where $T\in\mathrm{Irr}(Z\cap V)$ (resp.$T'\in\mathrm{Irr}(Z\cap W)$) with $\overline{T'}^{V_{\mathrm{Zar}}}=T$, $\overline{T}^{X_{\mathrm{Zar}}}=Z$. Indeed, since both sides are additive, we can assume that $D|_V = f\in R(\OO_X(V))$. Let $Y\subset V$ be a closed analytic subspace determined by $f\in \OO_X(V)$, then our claim is from Lemma~\ref{multiplicity of a coherent sheaf}.
	
	To show the lemma, we apply Lemma~\ref{key lemma for a local property to be global on irreducible subsets}. Let $m=\mathrm{mult}_{T}(D|_V)$ for some affinoid domain $V\subset X$ with $Z\cap V\not=\emptyset$, $D|_{V}\in K'_X(V)$, where $T\in\mathrm{Irr}(Z\cap V)$ with $\overline{T}^{X_{\mathrm{Zar}}}=Z$. For $V$ given as before, we say that an irreducible component $T\in\mathrm{Irr}(Z\cap V)$ satisfies $P$ if $\mathrm{mult}_{T}(D|_V)=m$. After replacing $X$ by $Z$, from our claim, we see that $P$ satisfies the hypothesis in Lemma~\ref{key lemma for a local property to be global on irreducible subsets}. Then there are Zariski-closed subset $Z_P^+, Z_P^-$ of $Z$ such that 
	\[Z_P^+\cap V = \bigcup\limits_{\substack{T\in \mathrm{Irr}(Z\cap V),\\ T \text{ satisfies $P$}}}T,\] \[Z_P^-\cap V = \bigcup\limits_{\substack{T\in \mathrm{Irr}(Z\cap V),\\ \text{$T$ doesn't satisfy $P$}}}T,\]
	and $Z=Z_P^+\cup Z_P^-$. Since $Z$ is irreducible, and there is some $T \subset Z_P^+$, so $Z=Z_P^+$. This implies the lemma.
\end{proof}

\begin{definition} 	\label{class group}
	For any $D\in \Div(X)$, we set
	\[[D]\coloneqq\sum\limits_{\substack{Z\in \overline{\mathrm{Irr}(X)}\\ \codim(Z,X)=1}}\mathrm{mult}_Z(D)[Z]\in Z^1(X),\]
	called the \emph{Weil divisor associated to $D$}.
	In particular, for any $f\in K^*(X)$, we denote
	$(f)\coloneqq[\mathrm{div}(f)]\in Z^1(X)$. Such a divisor $(f)$ is called a \emph{principal divisor}. The set of principal divisors $\mathrm{Rat}^1(X)$ form a subgroup of $Z^1(X)$. We denote the quotient of $Z^1(X)$ by the subgroup of principal divisors by $\Cl(X) \coloneqq Z^1(X)/\mathrm{Rat}^1(X)$, called the \emph{class group} of $X$. We say that two divisors $Z, Z'$ are \emph{rationally equivalent} and write $Z\sim_{\mathrm{rat}} Z'$ if they have the same class in $\Cl(X)$.
\end{definition}

Recall, a $K$-analytic space $X$ is \emph{regular} at $x\in X$ if there is a good analytic domain $V$ of $X$ containing $x$ such that $\OO_{V,x}$ is regular. We say $X$ is \emph{regular} if $X$ is regular at every point $x\in X$. This is equivalent to that for any affinoid domain $V\simeq \mathcal{M}(A)\subset X$, we have that $A$ is regular, see \cite[Lemma-Definition~2.4.1, Lemma~2.4.5]{ducros2018families}.

\begin{proposition} 	\label{prop:cartierdivisorandweildivisor}
	The map	$[\cdot]\colon \Div(X)\rightarrow Z^1(X)$ is a homomorphism 
	which sends effective divisors to positive cycles. This induces a homomorphism
	\[[\cdot]\colon \CaCl(X)\rightarrow \Cl(X).\]
	If $X$ is normal (resp. regular), then these two maps are injective (resp. isomorphic).
\end{proposition}
\begin{proof}
	It is easy to see that $[\cdot]\colon \Div(X)\rightarrow Z^1(X)$ is a homomorphism and induces $[\cdot]\colon \CaCl(X)\rightarrow \Cl(X)$. If $X$ is normal, by Lemma~\ref{lemma:localequivalence}, to show $[\cdot]\colon \Div(X)\rightarrow Z^1(X)$ is injective, we can assume $X$ is affinoid. For $D\in \Div(X)$ such that $\mathrm{mult}_Z(D) = 0$ for any $Z\in Z^1(X)$, we take affinoid domain $V\subset X$ with $Z\cap V\not=\emptyset$ and $D|_V\in K_X'(V)$. Then $D|_V\in \OO_X^*(V)$ since $\mathrm{mult}_T(D|_V)=0$ for any $Q\in Z^1(V)$. This implies that $D=0$. As for the quotient, if $[D]=(f)$ for some $f\in K^*_X(X)$, then $D=\mathrm{div}(f)$, this implies that $[\cdot]\colon \CaCl(X)\rightarrow \Cl(X)$ is injective. 
	
	Assume that $X$ is regular. To show that  $[\cdot]\colon \Div(X)\rightarrow Z^1(X)$ is surjective, we firstly assume that $X=\mathcal{M}(A)$ is affinoid and set $\mathcal{X}=\Spec(A)$. In this case,  $\Div(\mathcal{X})\simeq Z^1(\mathcal{X})$,  see \cite[Proposition~7.2.16]{liu2006algebraic}. Hence, we have a commutative diagram
	\[\xymatrix{\Div(\mathcal{X})\ar@{^{(}->}[r] \ar[d]_{\rotatebox{90}{$\sim$}}&\Div(X)\ar[d]^\rho\\ 
		Z^1(\mathcal{X})\ar[r]^\sim & Z^1(X)},\]
	so our claim holds for affinoid spaces. We can glue Cartier divisors on affinoid domains together by injectivity of $[\cdot]$. Hence $[\cdot]\colon \Div(X)\rightarrow Z^1(X)$ is surjective.
\end{proof}

\subsection{Flat pull-backs}

\

We have introduced Cartier divisors, cycles. Next we consider their pull-backs via flat morphisms.

Recall the definition of flatness in sense of \cite[Definition~4.1.8]{ducros2018families}, a morphism $f\colon Y\rightarrow X$ of $K$-analytic spaces is \emph{naively flat} if for any $y\in {Y}$, there exists a good analytic domain $V\subset Y$ containing $y$ and a good analytic domain $U\subset X$ containing $f(V)$ such that $\OO_{V,y}$ is flat over $\OO_{U,f(y)}$. We say $f$ is \emph{flat} if moreover $Y'\coloneqq Y\times_XX'\rightarrow X'$ is naively flat for any morphism $X'\rightarrow X$. If $f$ is flat, then $\OO_Y(V)$ is flat over $\OO_X(U)$ for any affinoid domains $V\subset Y$ and $U\subset X$ with $f(V)\subset U$. The converse is not true in general unless $f$ is locally finite. Note that for any analytic domain $V$ of $X$, the natural morphism $V\hookrightarrow X$ is flat.

\begin{definition}
	A morphism $f\colon Y\rightarrow X$ of $K$-analytic spaces has \emph{relative dimension $r$} if for any $Z\in \overline{\mathrm{Irr}(X)}$, $f^{-1}(Z)=\emptyset$ or any irreducible component $Z'$ of $f^{-1}(Z)$ has $\dim_KZ'=\dim_KZ+r$.
\end{definition}
\begin{remark}
The notion of relative dimension $r$ is an analogue of the one in algebraic geometry, see \cite[B.2.5]{fulton1998intersection}. Our definition is different from the one in \cite[\S 1.4.13]{ducros2018families}. We don't assume that such morphisms are surjective.
\end{remark}

\begin{lemma} \label{lemma:relative dimension r}
	Let $f\colon Y\rightarrow X$ be a flat morphism of $K$-analytic spaces. Then $f$ has relative dimension $r$ if and only if $Y_x=\emptyset$ or $Y_x$ is of pure dimension $r$ for any $x\in X$. In particular, if $f\colon Y\rightarrow X$ is flat with $X, Y$ equidimensional, then f has relative dimension $\dim_KY-\dim_KX$.
\end{lemma}
\begin{proof}
	We apply \cite[Lemma~4.5.11]{ducros2018families} saying that $\dim_yY=\dim_yY_x+\dim_xX$ for any $y\in Y_x$.

	Assume that $f$ has relative dimension $r$. If $x\in X$ such that $Y_x\not=\emptyset$, then for any $Z\in {\mathrm{Irr}(X)}$ containing $x$, we have $\dim_Kf^{-1}(Z)-\dim_KZ=r$. This implies that $\dim_yY_x=\dim_yY-\dim_xX=r$ for any $y\in Y_x$ since $\dim_xX=\max\limits_{x\in Z\in\mathrm{Irr}(X)}\{\dim_KZ\}$.
	
	Conversely, for any $Z\in \overline{\mathrm{Irr}(X)}$ with $f^{-1}(Z)\not=\emptyset$, without loss of generality, we can assume that $Z=X$. We take $y\in Y$ and $x=f(y)$. Then $\dim_yY=\dim_xX+\dim_yY_x=\dim_KX+r$. This implies that $f$ has relative dimension $r$.
	
	If $X, Y$ are equidimensional, then $\dim_yY_x=\dim_yY-\dim_xX$ implies that $Y_x$ is of equidimension for any $y\in Y, x=f(x)$.
\end{proof}

\begin{definition}\label{def:flatpullback}
	Let $f\colon Y\rightarrow X$ be a flat morphism of $K$-analytic spaces. 
	\begin{enumerate}
		\item [(1)] The canonical morphism $f^\#\colon \OO_X\rightarrow f_*\OO_Y$ extends to a morphism $f^\#\colon K_X^*/\OO_X^*\rightarrow f_*(K_Y^*/\OO_X^*)$, then we have a homomorphism
		\[f^*\colon \Div(X)\rightarrow \Div(Y).\]
		This will induce a homomorphism $f^*\colon \CaCl(X)\rightarrow \CaCl(Y)$.
		\item [(2)] Assume that $f$ has relative dimension $r$. For any integral closed subspace $Z\subset X$ of pure codimension $k$, we set 
		\[f^*[Z]\coloneqq [f^{-1}(Z)]\in Z^k(Y).\] 
		This extends by linearity to a pull-back homomorphism $f^*\colon Z^k(X)\rightarrow Z^k(Y)$. 
	\end{enumerate}
\end{definition}
\begin{remark} \label{remark:flatpullback}
The flat pull-backs are functorial and we have a commutative diagram
		\[\xymatrix{\Div(X)\ar[r]^{f^*}\ar[d]_{[\cdot]}& \Div(Y)\ar[d]^{[\cdot]}\\
			Z^1(X) \ar[r]_{f^*}& Z^1(Y)}.\]
\end{remark}

\begin{proposition}	\label{pull-back of cycles from coherent sheaf}
	Let $f\colon Y\rightarrow X$ be a flat morphism of $K$-analytic spaces that has relative dimension $r$. For a coherent sheaf $\mathcal{F}$ on $X$  with $\codim(\Supp(\mathcal{F}),X)\geq k$, we have $\codim(\Supp(f^*\mathcal{F}),X)\geq k$ and
	\[[f^*\mathcal{F}]^k=f^*[\mathcal{F}]^k.\]
	In particular, if $Z$ is a closed analytic subspace of $X$ of pure codimension $k$, then $f^*[Z] = [f^{-1}(Z)]$.
\end{proposition}
\begin{proof}
	We can reduce the statement to the case of affinoid spaces by \cref{lemma:localequivalence}, then the proposition from the analogue result in scheme theory by \cref{remark:cycle associated to a coherent sheaf}. For the result in scheme theory, see proof of \cite[Lemma~42.14.4~(2)]{stacks-project}.
\end{proof}

\subsection{Proper push-forward of cycles}

\

For an affinoid space $X=\mathcal{M}(A)$, it may happen that $\dim_{\mathrm{Krull}}A< \dim_KX$. In order to avoid this dimension problem, we assume that all $K$-analytic spaces (including affinoid domains) in this subsection are strict. In this case $\dim_{\mathrm{Krull}}A=\dim_KX$.

Recall a theorem of Kiehl. 

\begin{theorem}[\cite{berkovich1990spectral}~Proposition~3.3.5] \label{theorem:Kiehl}
	Let $f\colon Y\rightarrow X$ be a proper morphism of $K$-analytic spaces, and $\mathcal{F}$ a coherent $\OO_Y$-module. Then $R^nf_*\mathcal{F}, n\geq 0$, are coherent $\OO_X$-modules. 
	In particular, we have Remmert's mapping theorem, saying that $f(Y)$ is an Zariski-closed subset of $X$.
\end{theorem}

We have the following equivalent conditions.

\begin{lemma} \label{lemma:equivalentconditionforsurjectivefinite}
	Let  $f\colon Y\rightarrow X$ be a morphism of separated, strictly $K$-analytic spaces. Then the following statements are equivalent.
	\begin{enumerate}
		\item [(i)] $f$ is surjective and finite.
		\item[(ii)] $f$ is proper, and for any $x\in X$, $\dim_{\mathscr{H}(x)}Y_x=0$. 
		\item[(iii)] $f$ is proper, and for any rigid point $x\in X$, $Y_x\not=\emptyset$ has finite rigid points as an $\mathscr{H}(x)$-analytic space.
	\end{enumerate}
\end{lemma}
\begin{proof}
	Obviously, (i) implies (ii), and (ii) implies (iii). 
	It remains to show that (iii) implies (i). The separatedness ensure that $X, Y$ are also rigid $K$-analytic spaces, see \cite[Theorem~1.6.1]{berkovich1993etale}. Then the result is from \cite[Corollary~9.6.6]{bosch1984nonarchimedean} and \cite[Proposition~3.3.2]{berkovich1990spectral}.

\end{proof}

\begin{art} \label{abhyankar points}
	Recall an Abhyankar point is a point $x\in X$ such that $d_K(x) = \dim_{x}X$. If $X$ is irreducible, then $\overline{\{x\}}^{X_{\zar}}=X$. If $X$ is reduced, then $\OO_{X,x}$ is a field. Indeed, by \cite[Corollary~3.2.9]{ducros2018families}, we have that $\dim_KX=\dim_xX\leq \dim_K\overline{\{x\}}^{X_{\zar}}\leq\dim_KX$, so $\overline{\{x\}}^{X_{\zar}}=X$ if $X$ is irreducible. By \cite[Example~3.2.10]{ducros2018families}, we have that $\OO_{X,x}$ is Artinian, so is a field if $X$ is reduced. 
\end{art}

\begin{proposition} \label{proposition:proper equidimension}
	Let $f\colon Y\to X$ be a surjective, proper morphism of irreducible, separated, strictly $K$-analytic spaces. Then the following are equivalent.
	\begin{enumerate}
		\item [(i)] $\dim_K(X)=\dim_K(Y)$. 
		\item [(ii)] There is a non-empty Zariski-open subset $U\subset X$ such that $f\colon f^{-1}(U)\to U$ is finite.
		\item[(iii)] There is an Abhyankar point $x\in X$ such that $\dim_{\mathscr{H}(x)}Y_x=0$.
		\item[(iv)] There is a point $x\in X$ such that $\dim_{\mathscr{H}(x)}Y_x=0$.
	\end{enumerate}
	In this case, we say that $f$ is \emph{generically finite}.
\end{proposition}
\begin{proof}
	The equivalence of (i) and (iii) is \cite[Lemma~1.5.11]{ducros2018families}. It is obviously that (ii) implies (iii), and (iii) implies (iv). It is remains to show that (iii) implies (ii), and (iv) implies (iii).
	
	Assume that (iii) holds. By \cite[Theorem~10.7.5~(2)]{ducros2018families}, there is a maximal Zariski-open subset $U$ of $X$ such that $f^{-1}(U)\to U$ is flat and $\dim_{\mathscr{H}(x)}Y_x=0$ for any $x\in U$. The set $U$ is not empty by our assumption and \cite[Theorem~10.3.7]{ducros2018families}. By \cref{lemma:equivalentconditionforsurjectivefinite}, we know that $f^{-1}(U)\to U$ is finite. Hence (ii) holds.
	
	Assume that (iv) holds. By \cite[TH\'EOR\`EME~4.9]{ducros2007variation}, the set $U\coloneq\{y\in Y\mid \dim_yY_{f(y)}=0\}$ is Zariski-open in $Y$. Set $F=Y\setminus U$, then $f(F)\subset X$ is Zariski-closed. Notice that $f(F)=X$, i.e. $\dim_{\mathscr{H}(x)}Y_x\geq1$ for any $x\in X$, contracts with (iv), so $f(F)\neq X$. Hence there is an Abhyankar point $x\in X$ such that $\dim_{\mathscr{H}(x)}Y_x=0$, i.e. (iii) holds.
\end{proof}
\begin{remark} \label{remark:equidimension}
	If the equivalent conditions in \cref{proposition:proper equidimension} hold, then any Abhyankar point $x\in X$ is contained in $U$ in (ii) (see \cref{abhyankar points}), and the fiber $Y_x$ is consisted of finite isolated rigid points, and any $y\in Y_x$ is an Abhyankar point of $Y$ with  $\OO_{Y,y}=\OO_{Y_x,y}=\mathscr{H}(y)$. Indeed, if $\dim_KX=\dim_KY$, for any $y\in Y_x$, we have that $d_K(y) = d_K(x)+d_{\mathscr{H}(x)}(y)=\dim_KX+d_{\mathscr{H}(x)}(y)\leq \dim_KY$, then $d_K(y) = \dim_KY=\dim_K(Y,y)$ and $d_{\mathscr{H}(x)}(y)=0$. Hence $y\in Y$ is an Abhyankar point and $\OO_{Y,y}=\mathscr{H}(y)$. Since $Y_x$ is proper over $\mathscr{H}(x)$, so $Y_x$ is consisted of finite rigid points. The points are isolated since $Y$ is separated. On the other hand, since $\mathscr{H}(y)=\OO_{Y,y}\to\OO_{Y_x,y}$ is regular by \cite[Theorem~6.3.7]{ducros2018families}, so $\OO_{Y_x,y}$ is a field and $\OO_{Y_x,y}=\mathscr{H}(y)$.
\end{remark}

A similar result of the following lemma is given in \cite[\S~2.6]{gubler1998local}.

\begin{lemma} 	\label{lemma:welldefinedofdegree}
	Let $f\colon Y\rightarrow X$ be a surjective, proper morphism of integral, separated, strictly $K$-analytic spaces with $\dim_K(X)=\dim_K(Y)$. Let $x\in X$ be an Abhyankar point. Then the number 
		\[\deg(Y/X)\coloneqq\sum\limits_{y\in Y_x}[\mathscr{H}(y)\colon\mathscr{H}(x)]\]
		is independent of the choice of $x$, called the \emph{generic degree of $f$}.
\end{lemma}
\begin{proof}
	By \cref{proposition:proper equidimension} and \cref{remark:equidimension}, we can assume that $f\colon Y\to X$ is finite. 
	For any (strictly) affinoid domain $V\subset X$ and $T\in \mathrm{Irr}(V)$, set 
			\[d(V,T)\coloneq\sum\limits_{\substack{Q\in\mathrm{Irr}(f^{-1}(V))\\ f(Q)=T}}[\Frac(A_Q)\colon \Frac(A_T)].\]
	where $A_T, A_Q$ are the affinoid algebras corresponding to $T, Q$ with reduced structure. If $x\in T$, set $V=\mathcal{M}(A)$ and $f^{-1}(V)=\mathcal{M}(B)$. Notice that $x\in T$ is mapped to the generic point of $\Spec(A_T)$ via $T\to \Spec(A_T)$, so $\Frac(A_T)\subset \mathscr{H}(x)$, and 
	\begin{align*}
		\bigoplus_{y\in Y_x}\mathscr{H}(y)=&B\otimes_A\mathscr{H}(x) \\
		=&  (B\otimes_A\Frac(A_T))\otimes_{\Frac(A_T)}\mathscr{H}(x)\\
		=& \bigoplus\limits_{\substack{Q\in\mathrm{Irr}(f^{-1}(V))\\ f(Q)=T}}\Frac(A_Q)\otimes_ {\Frac(A_T)}\mathscr{H}(x).
	\end{align*} Hence $d(V,T) = \sum\limits_{y\in Y_x}[\mathscr{H}(y)\colon\mathscr{H}(x)]$. It remains to show that $d(V,T)$
	is independent of the choices of $V$ and $T$.
	
We first consider the case where $V, W\subset X$ are affinoid domains with $W\subset V$, and any $T\in \mathrm{Irr}(V)$, $T'\in \mathrm{Irr}(W)$, and show that $d(V,T)=d(W,T')$.
This is in fact from \cref{multiplicity of a coherent sheaf} and \cref{pull-back of cycles from coherent sheaf} for affinoid case. Let $V=\mathcal{M}(A), f^{-1}(V) = \mathcal{M}(B)$ and $W=\mathcal{M}(A')$, then $f^{-1}(W)=\mathcal{M}(B')$, where $B'=A'\otimes_AB$. Let $\mathcal{F}$ be the corresponding coherent sheaf associated to $B$ as an $A$-module on $V$, and $i\colon W\rightarrow V$ the canonical morphism, then 
\[[\mathcal{F}]^0= \sum\limits_{T\in\mathrm{Irr}(V)}(\sum\limits_{\substack{Q\in\mathrm{Irr}(f^{-1}(V))\\ f(Q)=T}}[\Frac(A_Q)\colon \Frac(A_T)])[T],\]
and we know that $\sum\limits_{\substack{Q\in\mathrm{Irr}(f^{-1}(V))\\ f(Q)=T}}[\Frac(A_Q)\colon \Frac(A_T)]$ is independent of the choice of $T$ by \cref{multiplicity of a coherent sheaf}. We also have
\[i^*[\mathcal{F}]^0=\sum\limits_{T\in\mathrm{Irr}(V)}(\sum\limits_{\substack{Q\in\mathrm{Irr}(f^{-1}(V))\\ f(Q)=T}}[\Frac(A_Q)\colon \Frac(A_T)])\sum\limits_{T'\in\mathrm{Irr}(T\cap W)}[T'],\]
\[[i^*\mathcal{F}]^0= \sum\limits_{T'\in\mathrm{Irr}(W)}(\sum\limits_{\substack{Q'\in\mathrm{Irr}(f^{-1}(W))\\ f(Q')=T'}}[\Frac(A_{Q'})\colon \Frac(A_{T'})])[T'].\] By \cref{pull-back of cycles from coherent sheaf}, we have that $i^*[\mathcal{F}]=[i^*\mathcal{F}]$. For any $T'\in \mathrm{Irr}(W)$, we compare the coefficients of $[T']$, then $d(V,T)=d(W,T')$.

To show that $d(V,T)$
is independent of the choices of $V$ and $T$, we apply Lemma~\ref{key lemma for a local property to be global on irreducible subsets}. Let $m=d(V_0,T_0)$ for some affinoid domain $V_0\subset X$, where $T_0\in\mathrm{Irr}(V_0)$. For $V$ given as before, we say that an irreducible component $T\in\mathrm{Irr}(V)$ satisfies $P$ if $d(V,T)=m$. By our discussion above, we see that $P$ satisfies the hypothesis in Lemma~\ref{key lemma for a local property to be global on irreducible subsets}. Then there are Zariski-closed subset $X_P^+, X_P^-$ of $X$ such that 
\[X_P^+\cap V = \bigcup\limits_{\substack{T\in \mathrm{Irr}(V),\\ T \text{ satisfies $P$}}}T,\] \[X_P^-\cap V = \bigcup\limits_{\substack{T\in \mathrm{Irr}(V),\\ \text{$T$ doesn't satisfy $P$}}}T,\]
and $X=X_P^+\cup X_P^-$. Since $X$ is irreducible, and there is some $T \subset X_P^+$, so $X=X_P^+$. This proves (3).
\end{proof}

With the lemmas above, we have the following definition.

\begin{definition}
	Let  $f\colon Y\rightarrow X$ be a proper morphism of separated, strictly $K$-analytic spaces. For any irreducible closed subspace $Z$ of $Y$, the image $f(Z)$ is a Zariski-closed subset of $Y$. We set 
	\[\deg(Z/f(Z))\coloneqq\begin{cases}
		\text{the generic degree of $f\colon Z\rightarrow f(Z)$}& \text{ if $\dim_Kf(Z) = \dim_KZ$;}\\
		0 & \text{ if $\dim_Kf(Z)<\dim_KZ$}
	\end{cases}\]
	Define $f_*[Z]\coloneqq \deg(Z/f(Z))[f(Z)]$, then extends $f_*$ linearly to a homomorphism (of gradding groups)
	\[f_*\colon Z_*(Y)\to Z_*(X).\]
\end{definition}
\begin{remark}
For $Z$ above, we know that $f(Z)$ with the reduced subspace structure is the Zariski image of $Z\rightarrow X$ by \cref{lemma:zariskiimageforreducedspace}.
\end{remark}

We can easily prove the following lemma.

\begin{lemma}
	Let $f\colon Y\rightarrow X$ and $g\colon Z\rightarrow Y$ be proper morphism of separated strictly $K$-analytic spaces. Then $g_*\circ f_*=(g\circ f)_*$.
\end{lemma}

\begin{proposition} \label{prop:pushforwordofcoherentsheaves}
	Let $f\colon Y\rightarrow X$ be a proper morphism of separated strictly $K$-analytic spaces.
	\begin{enumerate}
		\item [(1)] Let $Z\subset Y$ be a closed subspace with $\dim_KZ\leq k$. Then
		\[f_*[Z]_k=[f_*\OO_Z]_k.\]
		\item[(2)] Let $\mathcal{F}$ be a coherent sheaf on $Y$ such that $\dim_K(\Supp(\mathcal{F}))\leq k$. Then
		\[f_*[\mathcal{F}]_k=[f_*\mathcal{F}]_k.\]
	\end{enumerate}
\end{proposition}
\begin{proof}
	Obviously, it suffices to show (2). By \cref{lemma:coherentsheafonsupport}, there is a coherent sheaf $\mathcal{G}$ on $Z\coloneq\Supp(\mathcal{F})$ such that $\mathcal{F}=i_*\mathcal{G}$. Let $Z'$ be the Zariski image of $Z\rightarrow X$. Notice that $f(Z)=Z'$ by properness of $f$ and \cref{lemma:zariskiimageforquasicompactmorphism}. So we have the following commutative diagram
	\[\xymatrix{Z\ar@{^{(}->}[r]\ar[d]_{f|_Z}& Y\ar[d]^f\\ Z'\ar@{^{(}->}[r]& X}.\]
	By functorial property of push-forward, it suffices to show $(f|_Z)_*[\mathcal{G}] = [(f|_Z)_*\mathcal{G}]$. So we can assume that 
	$\dim_KY=k$ and $f\colon Y\to X$ is proper and dominant. Moreover, we can assume that $\dim_KX=k$ (if $\dim_KX<k$, then $f_*[\mathcal{F}]=0=[f_*\mathcal{F}]_k$) and that $X$ is affinoid since both sides can be calculated locally on $X$. 
	
	We write
	\[f_*[\mathcal{F}]_k = \sum\limits_{W}n_W[W] \ \ \text{ and } \ \ [f_*\mathcal{F}]_k = \sum\limits_{W}m_W[W]\]
	where $W$ runs through irreducible component of $X$ of dimension $k$. For a fixed irreducible component $W$, to show $n_W=m_W$, it suffices to show that $(f_*[\mathcal{F}]_k)|_V = ([f_*\mathcal{F}]_k)|_V$ for some affinoid domain $V\subset X$ with $V\cap W\not=\emptyset$.  We can take 
	there is a Zariski-open subsets $U\subset X$ such that $U\cap W'=\emptyset$ and $U\cap f(T)=\emptyset$ for any irreducible component $W'$ of $X$ which is distinct from $W$, and any irreducible component $T$ of $Y$ which doesn't dominate $W$. We can take an affinoid domain of $U$ which is irreducible. So we can assume that $X=\mathcal{M}(A)$ is irreducible and each irreducible component of $Y$ (here $Y$ might be empty) dominates some irreducible component of $X$. 
	Moreover, we can assume that $Y$ is finite over $X$. Indeed, for any irreducible component $T$ of $Y$, we have that $T\to X$ is proper, and $\dim_KT=\dim_KX$. By \cref{proposition:proper equidimension}~(ii), there is an irreducible affinoid domain $V$ (in some Zariski-open subset of $X$) such that $\dim_{\mathscr{H}(x)}Y_x=0$ for any $x\in V$. By \cite[Corollary~3.3.8]{berkovich1990spectral}, $f^{-1}(V)\to V$ is finite, and we can replace $X$ by $V$. So we reduce to the case where $Y, X$ is affinoid and $f$ is finite. This is an algebraic result, see the last part of the proof of \cite[Lemma~41.13.3]{stacks-project}.
\end{proof}

\begin{proposition}\label{prop:flatbasechangeofcycles}
	Let 
	\[\xymatrix{Y'\ar[r]^{g'}\ar[d]_{f'}& Y \ar[d]^{f}\\
	X'\ar[r]_{g}& X}\]
be a Cartesian diagram of separated, strictly $K$-analytic spaces. Assume that $f$ is proper, and that $g$ is flat, has of relative dimension $r$. Then $f'$ is proper, $g'$ is flat and has relative dimension $r$. Moreover, $g^*\circ f_*= f'_*\circ g'^*$ on $Z^*(Y)$.
\end{proposition}
\begin{proof}
	The morphism $f'$ is proper by \cite{berkovich1993etale}, and $g'$ is flat by definition. For any $y'\in Y'$, set $y=g'(y')\in Y$. Notice that
	\[Y_y'=Y'\times_Y\mathcal{M}(\mathscr{H}(y))=X'\times_X\mathcal{M}(\mathscr{H}(y))
	=X'_x\times_{\mathcal{M}(\mathscr{H}(x))}\mathcal{M}(\mathscr{H}(y))\]
	where $x=f(y)\in X$. Since $g$ is flat and has relative dimension $r$, we have that $X_x'=\emptyset$ or $\dim_{\mathscr{H}(x)}(X_x)=r$ by \cref{lemma:relative dimension r}. So $Y_y'=\emptyset$ or $\dim_{\mathscr{H}(y)}(Y_y)=r$ by strictness. Hence $g$ has relative dimension $r$ by \cref{lemma:relative dimension r}.
	
	For the equality, notice that it holds if $f$ is a closed immersion. In general, for any $\alpha\in Z^*(Y)$, to show $g^*(f_*\alpha)= f_*'(g'^*(\alpha))$, we can assume that $\alpha = [Y]$ and it is irreducible. Moreover, we can assume that $X=f(Y)$. 
	
	If $\dim_K X<\dim_KY$, then left-handed side is $0$. For any Abhyankar point $x'\in X'$, let $x=g(x')$. We have \[Y'_{x'} = Y'\times_{X'}\mathcal{M}(\mathscr{H}(x')) =Y\times_{X}\mathcal{M}(\mathscr{H}(x')) = Y_{x} \times_{\mathcal{M}(\mathscr{H}(x))}\mathcal{M}(\mathscr{H}(x')).\] 
	By \cref{proposition:proper equidimension}, we have that $\dim_{\mathscr{H}(x)}Y_x\geq1$. So $\dim_{\mathscr{H}(x')}Y_{x'}'\geq 1$. Notice that $X'$, $Y'$ are equidimensional, so $Y_{x'}'$ is equidimensional by \cite[Lemma~1.5.11]{ducros2018families}. By considering each irreducible component of $X'$ and $Y'$, and by \cref{proposition:proper equidimension}, we have that $\dim_{K}X'<\dim_KY'$. So $f'^*([Y'])=0$.
	
	If $\dim_K X=\dim_KY$, then there is a Zariski-open subset $U$ of $X$ such that $f\colon f^{-1}(U)\rightarrow U$ is finite. Notice that $g^{-1}(U)\subset X'$ (resp. $g'^{-1}(f^{-1}(U))$) is Zariski-dense since $g$ (resp. $g'$) is flat and has relative dimension $r$, we can assume that $U=X$, i.e. $f\colon Y\to X$ is finite. By \cref{lemma:localequivalence}, it suffices to consider the affinoid case, write $Y=\mathcal{M}(B), X=\mathcal{M}(A)$. Notice that $f_*\OO_Y=\widetilde{B}$ is the coherent sheaf on $X$ associated to $B$. We claim that $g^*f_*\OO_Y=f'_*\OO_{Y'}$. Indeed, for any affinoid domain $V=\mathcal{M}(C)\subset X'$, we have that
	\[(g^*f_*\OO_Y)(V) = C\otimes_AB= \OO_{Y'}(V\times_XY) = \OO_{Y'}(f^{-1}(V))=f_*'\OO_{Y'}(V),\]
	this proves our claim.
	Then $$g^*f_*[Y]=g^*[f_*\OO_Y]=[g^*f_*\OO_Y]=[f'_*\OO_{Y'}] = f'_*[\OO_{Y'}]=f'_*[g'^*\OO_Y]=f'_*g'^*[Y],$$
	where the first and the fourth equality are from \cref{prop:pushforwordofcoherentsheaves}, the second and the last equality are from \cref{pull-back of cycles from coherent sheaf}, and the third is our claim. This proves the proposition.
\end{proof}



\section{Proper intersection and intersection multiplicities}

\label{properintersection}

In this section, we fix an equidimensional, regular $K$-analytic space $X$.

\subsection{Proper intersection}
\begin{lemma} \label{lemma: codimension inequality}
	Let $Y, \widetilde{Y}\in \overline{\mathrm{Irr}(X)}$. Then for every irreducible component $Z$ of $Y\cap \widetilde{Y}$, we have
	\[\codim(Z,X)\leq \codim(Y,X)+\codim(\widetilde{Y},X).\]
\end{lemma}
\begin{proof}
	The proof is based on the corresponding result in scheme theory.
	We can assume that $X$ is irreducible. For any affinoid domain $V\subset X$, we have that
	$\codim(T, V)= \codim(Y,X)$, where $T$ is a irreducible component of $V\cap Y$. Then we can apply the corresponding result in scheme theory.
\end{proof}

\begin{definition} 	\label{proper intersection on analytic spaces}
	\begin{enumerate}
		\item [(1)] Let $Y, \widetilde{Y}\in \overline{\mathrm{Irr}(X)}$. We say that $Y$ and $\widetilde{Y}$ \emph{intersect properly} if $\codim(Z,X)\geq \codim(Y,X)+\codim(\widetilde{Y},X)$.
		\item[(2)] Let $\alpha=\sum\limits_{i\in I}n_i[Y_i]\in Z^s(X)$ and $\beta=\sum\limits_{j\in J}m_j[\widetilde{Y}_j]\in Z^r(X)$. We say that $\alpha$ and $\beta$ \emph{intersect properly} if $Y_i$ and $\widetilde{Y}_j$ intersect properly for all $i$ and $j$.
	\end{enumerate}
\end{definition}

\begin{lemma} \label{lemma:properlyintersectislocal}
	Let $Y, \widetilde{Y}\in \overline{\mathrm{Irr}(X)}$. Then the following statements are equivalent:
	\begin{enumerate}
		\item [(i)] $Y, \widetilde{Y}$ intersect properly;
		\item[(ii)] For any $x\in Y\cap \widetilde{Y}$, there is an affinoid domain $V$ containing $x$ such that any $Q\in \mathrm{Irr}(Y\cap V), \widetilde{Q}\in\mathrm{Irr}(\widetilde{Y}\cap V)$ intersect properly on $V$;
		\item [(iii)] For any affinoid domain $V$ with $Y\cap V$, $\widetilde{Y}\cap V \not=\emptyset$ and any $Q\in \mathrm{Irr}(Y\cap\mathrm{V}), \widetilde{Q}\in\mathrm{Irr}(\widetilde{Y}\cap V)$, we have $Q$ and $\widetilde{Q}$ intersect properly.
	\end{enumerate} 
\end{lemma}
\begin{proof}
	For any affinoid domain $V\subset X$ with $Y\cap V = \emptyset$ and any $Q\in \mathrm{Irr}(Y\cap V)$, we have $\codim(Q, V)= \codim(Y,X)$. Then the lemma follows.
\end{proof}

\subsection{Multiplicities and intersect products}

\

In this subsection, we will apply the intersection theory on a regular catenary Noetherian scheme to define multiplicities. Another definition using Tor formula will be given in the next subsection.

Recall, on a regular, catenary Noetherian scheme $\mathcal{X}$, let $Q, \widetilde{Q}$ be irreducible closed subschemes with $\codim(Q,\mathcal{X})= s, \codim(\widetilde{Q},\mathcal{X}) =t$. Then intersection product of $Q, \widetilde{Q}$ is defined as 
\begin{align*}
Q\cdot \widetilde{Q}=\sum\limits_{T}e_T[T]\coloneqq\sum\limits_{i}(-1)^i[\mathrm{Tor}_i^{\OO_\mathcal{X},T}(\OO_Q,\OO_{\widetilde{Q}})]^{s+t}\in Z^{s+t}(X),
\end{align*}
i.e. 
\[e_T=e(\mathcal{X}, Q\cdot \widetilde{Q}, T) = \sum\limits_{i}(-1)^i\mathrm{length}_{\OO_{\mathcal{X},T}}(\mathrm{Tor}_i^{\OO_{\mathcal{X},T}}(\OO_{Q,T},\OO_{\widetilde{Q},T}))\]
where $T$ runs through $\mathrm{Irr}(Q\cap \widetilde{Q})$ with $\codim(T,\mathcal{X})=s+t$, and $\OO_{\mathcal{X},T}$ (resp. $\OO_{Q,T}$, resp. $\OO_{\widetilde{Q},T}$) denotes the local ring of $\mathcal{X}$ (resp. $Q$, resp. $\widetilde{Q}$) at the generic point of $T$. 

\begin{lemma} 	\label{multiplicities of proper intersection}
		Let $Y, \widetilde{Y}\subset X$ be irreducible Zariski-closed subspaces with $\codim(Y, X)=s, \codim(\widetilde{Y},X) = t$. Assume that $Y$ and $\widetilde{Y}$ intersect properly. For any irreducible component $Z$ of $Y\cap \widetilde{Y}$ with $\codim(Z,X)=s+t$, and any affinoid domain $V\subset X$ with $Z\cap V\not=\emptyset$, we set
	\[e(X,Y\cdot\widetilde{Y},Z)\coloneqq\sum\limits_{Q,\widetilde{Q}}e(V,Q\cdot\widetilde{Q},T)\]
	where $T\in \mathrm{Irr}(Z\cap V)$ and $(Q, \widetilde{Q})$ runs through $\mathrm{Irr}(Y\cap V)\times \mathrm{Irr}(\widetilde{Y}\cap V)$ such that $T\in \mathrm{Irr}(Q\cap \widetilde{Q})$. Then $e(X,Y,\widetilde{Y},Z)$ is a positive integer which is independent of the choice of $V$ and $T$. We call $e(X,Y,\widetilde{Y},Z)$ the \emph{multiplicity of $Z$ on $Y\cap \widetilde{Y}$}.
\end{lemma}
\begin{proof}
	The idea of proof is similar with the proof of Lemma~\ref{multiplicity of a coherent sheaf} and Lemma~\ref{multiplicity of a divisor}. It is sufficient to show that 
	for any affinoid domain $V, W\subset X$ with $W\subset V$, $Z\cap W\not=\emptyset$, we have that 
	\[\sum\limits_{Q,\widetilde{Q}}e(V,Q\cdot\widetilde{Q},T) = \sum\limits_{Q',\widetilde{Q}'}e(W,Q'\cdot\widetilde{Q}',T')\]
	where $T\in \mathrm{Irr}(Z\cap V)$, $(Q, \widetilde{Q})$ runs through $\mathrm{Irr}(Y\cap V)\times \mathrm{Irr}(\widetilde{Y}\cap V)$ such that $T\in \mathrm{Irr}(Q\cap \widetilde{Q})$, and $T', Q',\widetilde{Q}'$ is given similarly with $\overline{T'}^{V_{\mathrm{Zar}}}=T$, $\overline{T}^{X_{\mathrm{Zar}}}=Z$. Let $V=\mathcal{M}(A), W=\mathcal{M}(B)$ and $f\colon \Spec(B)\to \Spec(A)$ is the morphism of schemes given by $W\subset V$. 
	In the following, we view every irreducible subset in the corresponding affine schemes.
	We fix a pair $(Q, \widetilde{Q})$. Let $f^*[Q]=\sum\limits_{i=1}^{m}[Q_i'], f^*[\widetilde{Q}]=\sum\limits_{j=1}^{\widetilde{m}}[\widetilde{Q}_{j}']$, $[Q]\cdot[\widetilde{Q}]=\sum\limits_{p=1}^ke(V,Q\cdot \widetilde{Q}, T_p)[T_p]$ with $T_1=T$, and $f^*[T_p] = \sum\limits_{q=1}^{l_q}[T_{pq}']$ with $T_{11}'=T'$. Notice that each coefficient of $[Q_i']$ in $f^*[Q]$ is $1$ by \cref{multiplicity of a coherent sheaf} (notice that the flatness of $A\to B$ is used in the proof of \cref{multiplicity of a coherent sheaf}), similar for $f^*[\widetilde{Q}]$ and $f^*[T_p]$. In scheme theory, we have \[f^*[Q]\cdot f^*[\widetilde{Q}] = f^*([Q]\cdot [\widetilde{Q}]),\]
	i.e. \[\sum\limits_{i,j}[Q_i']\cdot[\widetilde{Q}_{j}'] = \sum\limits_{i,j,p,q}e(W,Q_i'\cdot \widetilde{Q}_j',T_{pq}')[T_{pq}']=\sum\limits_{p,q}e(V,Q, \widetilde{Q}, T_p)[T_{pq}'],\]
	where $e(W,Q_i'\cdot \widetilde{Q}_j',T_{pq}')=0$ if $T_{pq}'\not\in \mathrm{Irr}(Q_i'\cap \widetilde{Q}_j')$. Comparing the coefficients of $[T_{11}']$, we have $e(V,Q\cdot\widetilde{Q},T)=\sum\limits_{i,j}e(W,Q_i'\cdot \widetilde{Q}_j',T')$.  When $(Q,\widetilde{Q})$ runs through $\mathrm{Irr}(Y\cap V)\times \mathrm{Irr}(\widetilde{Y}\cap V)$ such that $T\in \mathrm{Irr}(Q\cap \widetilde{Q})$, we have the equality we want.
\end{proof}

\begin{definition} 	\label{intersection product of cycles}
	Let $Y, \widetilde{Y}\subset X$ be irreducible Zariski-closed subspaces with $\codim(Y, X)=s, \codim(\widetilde{Y},X) = t$ that intersect properly. We define the \emph{intersection product of $Y$ and $\widetilde{Y}$} as
	\[[Y]\cdot [\widetilde{Y}]\coloneq\sum\limits_{Z}e_Z[Z]\in Z^{s+t}(X),\]
	where $Z$ runs through the set $\mathrm{Irr}(Y\cap \widetilde{Y})$ with $\codim(Z,X)=s+t$, and	$e_Z\coloneq e(X, Y\cdot \widetilde{Y}, Z)$ given in \cref{multiplicities of proper intersection}.
	
	In general, let $\alpha=\sum\limits_{i\in I}n_i[Y_i]\in Z^s(X)$ and $\beta=\sum\limits_{j\in J}m_j[\widetilde{Y}_j]\in Z^r(X)$. Assume that $\alpha$ and $\beta$ intersect properly. We define 
	\[\alpha\cdot \beta\coloneqq\sum\limits_{i,j}n_im_j[Y_i]\cdot [\widetilde{Y}_j].\]
\end{definition}

From the associativity of intersections in scheme theory, we have the associativity for our definition.

\begin{corollary}
	Let $Y, \widetilde{Y}, \widetilde{\widetilde{Y}}$ be irreducible Zariski-closed subspaces of $X$. Assume that $Y, \widetilde{Y}, \widetilde{\widetilde{Y}}$ intersect properly pairwise and that $\codim(Y\cap\widetilde{Y}\cap \widetilde{\widetilde{Y}}, X) = \codim(Y, X)+\codim(\widetilde{Y}, X)+\codim( \widetilde{\widetilde{Y}}, X)$. Then
	\[[Y]\cdot([\widetilde{Y}]\cdot[\widetilde{\widetilde{Y}}]) = ([Y]\cdot[\widetilde{Y}])\cdot[\widetilde{\widetilde{Y}}]\]
	as cycles on $X$.
\end{corollary}
\begin{proof}
	This is from \cref{lemma:localequivalence} and the corresponding algebraic result, see \cite[Lemma~43.20.1]{stacks-project}.
\end{proof}

\begin{lemma}	\label{lemma:flatpullbackpreservesintersection}
	Let $f\colon X\to Y$ be a flat morphism of equidimensional, regular $K$-analytic spaces. Let $\mathcal{F}, \mathcal{G}$ be coherent sheaves on $Y$ with $\codim(\Supp(\mathcal{F}),X)\leq r, \codim(\Supp(\mathcal{G}),X)\leq s$, and $\codim(\Supp(\mathcal{F})\cap\Supp(\mathcal{G}),X)\geq r+s+\dim(Y)-\dim(X)$. In this case, the cycle $[f^*\mathcal{F}]^r$ and $[f^*\mathcal{G}]^s$ intersect properly and 
	\[f^*([\mathcal{F}]^r\cdot [\mathcal{G}]^s) = [f^*\mathcal{F}]^r\cdot[f^*\mathcal{G}]^s.\]
\end{lemma}
\begin{proof}
This is from \cref{lemma:localequivalence} and \cite[Lemma~43.21.1]{stacks-project} for regular, catenary Noetherian schemes.
\end{proof}

The lemma implies the following corollary directly.

\begin{corollary}
	Let $f\colon X\to Y$ be flat morphism of regular $K$-analytic spaces. Let $\alpha\in Z^r(Y), \beta\in Z^s(Y)$. Assume that $\alpha$ and $\beta$ intersect properly. Then $f^*\alpha$ and $f^*\beta$ intersect properly and $f^*(\alpha\cdot\beta) = f^*\alpha\cdot f^*\beta$.
	\label{intersection product preserved under flat pullback for analytic spaces}
\end{corollary}

\subsection{Intersection multiplicities using Tor formula}

\

We could define the multiplicities following the idea in \cite[Section~43]{stacks-project} using $\mathrm{Tor}_i^{\OO_X}(\mathcal{F},\mathcal{G})$.

It is not hard to see that $\mathrm{Tor}_i^{\OO_X}(\mathcal{F},\mathcal{G})$ is a coherent sheaf on $X$. Indeed, it suffices to consider the case where $X=\mathcal{M}(A)$ is affinoid, then $\mathcal{C}oh(X)\simeq \mathcal{C}oh(\Spec(A))$. Since $A$ is Noetherian, so we see that $\mathrm{Tor}_i^{\OO_X}(\mathcal{F},\mathcal{G})$ is a coherent sheaf on $X$.

We show the following result.

\begin{proposition}\label{prop:intersectionmultiplicityusingtorformula}
\begin{enumerate}
	\item [(1)] Let $Y, \widetilde{Y}$ be irreducible Zariski-closed subspaces of $X$ with $\codim(Y,X)=s, \codim(\widetilde{Y},X)=t$. Assume that $Y, \widetilde{Y}$ intersect properly. Then 
	\[[Y]\cdot [\widetilde{Y}]=\sum\limits_{i}(-1)^i[\mathrm{Tor}_i^{\OO_X}(\OO_{Y},\OO_{\widetilde{Y}})]^{s+t}.\]
	\item[(2)] Let $\mathcal{F}, \mathcal{G}$ be coherent sheaves on $X$ with $\codim(\Supp(\mathcal{F}),X)\geq s$, $\codim(\Supp(\mathcal{G}),X)\geq t$. Assume that $[\mathcal{F}]^s, [\mathcal{G}]^t$ intersecting properly. Then 
	\[[\mathcal{F}]^s\cdot[\mathcal{G}]^t = \sum\limits_{i}(-1)^i[\mathrm{Tor}_i^{\OO_X}(\mathcal{F},\mathcal{G})]^{s+t}.\]
\end{enumerate} 
\end{proposition}
\begin{proof}
	Obviously, (2) implies (1). 
	By \cref{lemma:localequivalence}, \cref{lemma:properlyintersectislocal} and \cref{lemma:flatpullbackpreservesintersection}, we can assume that $X$ is strictly affinoid. Then this is \cite[Lemma~43.19.4]{stacks-project} for regular, catenary Noetherian schemes.
\end{proof}

\section{Projection formula}

\label{projectionformula}

For a $K$-analytic space $X$, we denote $D(\mathcal{C}oh(X))$ the derived category of $\mathcal{C}oh(X)$. 
We have the derived tensor product $\otimes^{\mathbf{L}}$ in $D(\mathcal{C}oh(X))$, see \cite[Definition~20.26.14]{stacks-project}. If $f\colon Y\to X$ is a morphism of $K$-analytic spaces, then we have a left derived functor
\[Lf^*\colon D(\mathcal{C}oh(X))\to D(\mathcal{C}oh(Y)),\]
see \cite[Section~21.18]{stacks-project}. If $f$ is proper, we have a right derived functor
\[Rf_*\colon D(\mathcal{C}oh(Y))\to D(\mathcal{C}oh(X)),\]
see \cite[Section~21.19]{stacks-project}. By adjointness of $(Lf^*, Rf_*)$, we have a morphism
	\[Rf_*(\mathcal{E})\otimes^{\mathbf{L}}_{\OO_X}\mathcal{F}\rightarrow Rf_*(\mathcal{E}\otimes^{\mathbf{L}}_{\OO_Y}Lf^*\mathcal{F}),\]
see \cite[Section~21.50]{stacks-project}. As \cite[Lemma~36.22.1]{stacks-project}, we have a similar result for $K$-analytic spaces.

\begin{lemma} 	\label{lemma:projectionformulainderivedcategory}
	Let $f\colon Y\to X$ be a proper morphism of strictly $K$-analytic spaces. Then for any $\mathcal{F}$ in $D(\mathcal{C}oh(X))$ and $\mathcal{E}$ in $D(\mathcal{C}oh(Y))$, the canonical morphism
	\[Rf_*(\mathcal{E})\otimes^{\mathbf{L}}_{\OO_X}\mathcal{F}\rightarrow Rf_*(\mathcal{E}\otimes^{\mathbf{L}}_{\OO_Y}Lf^*\mathcal{F})\]
	is an isomorphism.
\end{lemma}
\begin{proof}
The proof is similar with the proof of \cite[Lemma~36.22.1]{stacks-project}. We can assume that $X=\mathcal{M}(A)$ is affinoid. In this case, $D(\mathcal{C}oh(X))$ is the derived category of finitely generated $A$-modules, which is a subcategory of $D(A)$, the derived category of $A$-modules. We fix a coherent sheaf $\mathcal{E}$ on $Y$. For an object $M$ in $D(A)$, we say that $T(M)$ holds if the morphism 
\[Rf_*(\mathcal{E})\otimes^{\mathbf{L}}_{\OO_X}\widetilde{M}\rightarrow Rf_*(\mathcal{E}\otimes^{\mathbf{L}}_{\OO_Y}Lf^*\widetilde{M})\] is an isomorphism, where $\widetilde{M}$ is the corresponding sheaf of $M$ on $X$. 

If $M=\bigoplus\limits_{i}M_i$ and $T(M_i)$ holds, then $T(M)$ holds. Let $N\rightarrow L\rightarrow M\rightarrow N[1]$ be a distinguished triangle in $D(A)$. If $T$ holds for two of $N,L,M,$ then it holds for the third. Notice that $T(A[n])$ holds for any $n\in\Z$, where $A[n]$ is the $n$-th shift of $A$ in $D(A)$. Hence $T(M)$ holds for any object $M$ in $D(A)$, see \cite[Remark~15.59.11]{stacks-project}.
\end{proof}

\begin{theorem}[Projection formula] 	\label{theorem:projectionformula}
	Let $f\colon Y\to X$ be a flat, proper morphism of regular, separated, strictly $K$-analytic spaces. Let $\alpha\in Z^*(Y)$ and $\beta\in Z^*(X)$. Assume that $\alpha$ and $f^*\beta$ intersect properly. Then $f_*(\alpha)$ and $\beta$ intersect properly and 
	\[f_*(\alpha)\cdot\beta=f_*(\alpha\cdot f^*\beta).\]
\end{theorem}
\begin{proof}
Our proof is an analytic version of the proof of \cite[Lemma~43.22.1]{stacks-project}.
	
By \cref{lemma:properlyintersectislocal}, \cref{intersection product preserved under flat pullback for analytic spaces} and \cref{lemma:localequivalence}, we can assume that $X=\mathcal{M}(A)$ is affinoid and regular. Moreover, we assume that $\alpha=[Z], \beta=[W]$ for some irreducible closed subspaces of dimension $r$ and $s$.  

If $\dim_Kf(Z)\not=\dim_KZ$, then $f_*[Z]=0$, so $f_*[Z]$ and $[W]$ intersect properly. It suffices to show that $f_*([Z]\cdot f^*[W])=0$. We consider the morphism $Z\rightarrow f(Z)$, where $f(Z)$ is endowed with the reduced subspace structure. By \cref{proposition:proper equidimension}, every fiber of $Z\rightarrow f(Z)$ has dimension $\geq 1$. This implies that every fiber of the morphism $Z\cap f^{-1}(W)\rightarrow f(Z)\cap W$ has dimension $\geq 1$. Since $Z\cap f^{-1}(W)$ is equidimensional by our assumption, for any Abhyankar point $x\in f(Z)\cap W$, we have that $\dim_{\mathscr{H}(x)}Z_x=\dim_K(Z\cap f^{-1}(W))-\dim_xf(Z)\cap W\geq 1$ by \cite[Lemma~1.5.11]{ducros2018families}. So $\dim_K(Z\cap f^{-1}(W))>\dim_K(f(Z)\cap W)$, we conclude that $\dim_KT>\dim_Kf(T)$ for any irreducible component $T$ of $Z\cap f^{-1}(W)$. This implies what we want.

Assume that $\dim_Kf(Z)=\dim_KZ=r$. Let $T\subset f(Z)\cap W$ be an irreducible component, and $T_i\subset Z\cap f^{-1}(W)$, $i=1,\cdots, t$ the irreducible components of $Z\cap f^{-1}(W)$ dominating $T$. Since  $Z, f^{-1}(W)$ intersect properly, we have that \[\dim_KT\leq \dim_KT_i = \dim_KY-(\dim_KY-r+\dim_KX-s) = r+s-\dim_KX.\]
Then $f(Z)$ and $W$ intersect properly, and $f\colon T_i\to T$ is generically finite by \cref{lemma: codimension inequality} and \cref{proposition:proper equidimension}.
To show the equality, we follow the same idea of the proof of \cite[Lemma~42.23.1]{stacks-project}. Since $f$ is flat, by \cref{lemma:projectionformulainderivedcategory}, we have 
\[Rf_*(\OO_Z)\otimes^{\mathbf{L}}_{\OO_X}\OO_W\simeq Rf_*(\OO_Z\otimes^{\mathbf{L}}_{\OO_Y}f^*\OO_W).\]
So for any generic point $\xi\in \Spec(A)$ corresponding to an irreducible component of $f(Z)\cap W$, we have
\begin{align}
(f_*\mathrm{Tor}_i^{\OO_Y}(\OO_Z, f^*\OO_W))_\xi= (\mathrm{Tor}_i^{\OO_X}(f_*\OO_Z, \OO_W))_\xi.
\label{equalityforprojectionformula}
\end{align}
On the other hand, by \cref{prop:intersectionmultiplicityusingtorformula} and \cref{prop:pushforwordofcoherentsheaves}, we have
\begin{align*}
	f_*([Z]\cdot f^*[W])&=\sum\limits_i(-1)^if_*[\mathrm{Tor}_i^{\OO_Y}(\OO_Z, f^*\OO_W)]_{r+s-\dim_KY}\\
	&=\sum\limits_i(-1)^i[f_*\mathrm{Tor}_i^{\OO_Y}(\OO_Z, f^*\OO_W)]_{r+s-\dim_KY},
\end{align*}
\begin{align*}
	f_*[Z]\cdot [W] &= [f_*\OO_Z]\cdot [W]\\
	&=\sum\limits_i(-1)^i[\mathrm{Tor}_i^{\OO_X}(f_*\OO_Z, \OO_W)]_{r+s-\dim_KY}.
\end{align*}
Then $f_*([Z]\cdot f^*[W]) = f_*[Z]\cdot [W]$ by \cref{equalityforprojectionformula}.
\end{proof}

\section{GAGA}	
	
	\label{GAGA}
	
	It is natural to expect that our definitions of cycles, flat pull-backs, proper push-forwards and intersection products, for algebraic variety will be coincide with the ones in the intersection theory of algebraic varieties.

\begin{proposition} \label{prop:analyfication of cycles}
Let $X$ be an algebraic variety over $K$. Then we have an injective homomorphism $Z^*(X) \hookrightarrow Z^*(X^\an), \ \ [Y]\mapsto [Y^\an]$ which is an isomorphism if $X$ is proper. For a cycle $\alpha\in Z^*(X)$, we denote $\alpha^\an$ its image in $Z^*(X^\an)$. Moreover, the following properties hold.
 \begin{enumerate}
 	\item[(1)] For any affinoid domain $V$ contained in some affine open subset of $X^\an$, the diagram diagram commutes:
 	\[\xymatrix{Z^*(X)\ar[d]\ar[r]&Z^*(\mathcal{V})\ar[d]\\
 		Z^*(X^\an)\ar[r]&Z^*(V)},\]
 	where $\mathcal{V} = \Spec(\OO_{X^\an}(V))$. 
 	\item[(2)] Let $\alpha, \beta\in Z^*(X)$. Then $\alpha=\beta\in Z^*(X)$ (or $\alpha^\an=\beta^\an\in Z^*(X^\an)$) if and only if $i^*\alpha=i^*\beta\in Z^*(\mathcal{V})$ for any any affinoid domain $V$ contained in some affine open subset of $X^\an$, where  $\mathcal{V} = \Spec(\OO_{X^\an}(V))$ and $i\colon \mathcal{V}\to X$ is the canonical morphism. 
 \end{enumerate}
\end{proposition}
\begin{proof}
The map is obviously injective. When $X$ is proper, it suffices to show that every integral closed subspace $Z$ of $X^\an$ is algebraic. This is \cite[Proposition~3.4.11]{berkovich1990spectral}. 

(1) The diagram is directly from the definition of $[Y^\an]$ and \cref{remark:cycle associated to a coherent sheaf}.

(2) This is from the injectivity $Z^*(X) \hookrightarrow Z^*(X^\an)$, the commutative diagram in (1) and \cref{lemma:localequivalence}.
\end{proof}

\begin{proposition} \label{prop:GAGA}
Let $f\colon Y\to X$ be a morphism of algebraic varieties over $K$. We have the following hold.
\begin{enumerate}
	\item[(1)] Let $\mathcal{F}$ be a coherent sheaf on $X$. Then $[\mathcal{F}]^\an = [\mathcal{F}^\an]$.
	\item[(2)] We have a canonical homomorphism $\Div(X)\rightarrow \Div(X^\an), \ \ D\mapsto D^\an$ such that for any $D\in \Div(X)$, we have $[D]^\an = [D^\an]$. 
	\item[(3)] If $\varphi$ is flat with $X, Y$ equidimensional, let $\alpha\in Z^*(X)$, then $(\varphi^*(\alpha))^\an = (\varphi^\an)^*(\alpha^\an)$.
	\item[(4)] If $\varphi$ is proper and $\beta\in Z^*(Y)$, then $(\varphi_*(\beta))^\an = (\varphi^\an)_*(\beta^\an)$.
	\item[(5)] Let $\alpha, \beta\in Z^*(X)$. Then $\alpha, \beta$ intersect properly if and only if $\alpha^\an, \beta^\an\in Z^*(X^\an)$ intersect properly, and in this case, we have $(\alpha\cdot\beta)^\an= \alpha^\an\cdot\beta^\an$. 
\end{enumerate}
\end{proposition}
\begin{proof}
(1) Let $V=\mathcal{M}(B)\subset X^\an$ be an affinoid domain contained in some affine open subsets of $X^\an$. Then we have a canonical morphism $\varphi\colon \Spec(A)\to X$ which is flat by \cite[TH\'EOR\`EM~3.3]{ducros2009les}. It is sufficient to show that $[\mathcal{F}]^\an|_V=[\mathcal{F}^\an]|_V$. By the commutative diagram in \cref{prop:analyfication of cycles}~(1), we have that $[\mathcal{F}]^\an|_V = [\varphi^*\mathcal{F}]$ if we identify $Z^*(V)$ and $Z^*(\mathcal{V})$. On the other hand, by  \cref{remark:cycle associated to a coherent sheaf}, we have that $[\mathcal{F}^\an]|_V = [\mathcal{F}^\an|_V]$. So our claim holds.

(2) The homomorphism is given by the fact that $\mathcal{V}\rightarrow X$ is flat for any an affinoid domain $V=\mathcal{M}(A)\subset X^\an$ contained in some affine open subsets of $X^\an$, where $\mathcal{V}=\Spec(A)$. Then the compatibility on such affinoid domains will induce a divisor on $X$. The equality can be proved as (1).

(3) We take arbitrary affinoid domains $V=\mathcal{M}(A)\subset X^\an$ and $W=\mathcal{M}(B)\subset Y^\an$ such that $\varphi^\an(W)\subset V$ and $V$, $W$ are contained in some affine open subsets of $X^\an$, $Y^\an$ respectively. Let $\mathcal{V} = \Spec(A)$, $\mathcal{W} = \Spec(B)$. We have the following commutative diagram
\[\xymatrix{\mathcal{W}\ar[d]_-j\ar[r]^-{\widetilde{\varphi}}& \mathcal{V}\ar[d]^-i\\
Y\ar[r]_-\varphi&X}\]
Then
\[(\varphi^*(\alpha))^\an|_W = j^*\varphi^*(\alpha) = \widetilde{\varphi}^*i^*(\alpha)=(\varphi^\an|_W)^*(\alpha^\an|_V)=(\varphi^\an)^*(\alpha^\an)|_W.\]
here we identify the canonical isomorphisms $Z^*(V)\simeq Z^*(\mathcal{V})$ and $Z^*(W)\simeq Z^*(\mathcal{W})$. By \cref{lemma:localequivalence}, (3) follows.

(4) Since $\varphi$ is proper, we have that $\varphi^\an$ is proper. We may assume that $\beta$ is prime, moreover, assume that $X, Y$ are integral and $\beta=[Y]$, $\varphi$ is generically finite, surjective. After shrinking $X$, we can assume that $\varphi$ is finite. Moreover, we can assume that $X=\Spec(A)$ and $Y=\Spec(B)$ are affine. Let $V=\mathcal{M}(A')\subset X^\an$ be an affinoid domain, and $U=(\varphi^\an)^{-1}(V) = \mathcal{M}(A'\otimes_AB)$. Notice that $\Frac(B) = B\otimes_A\Frac(A)$. We consider the following diagram
\[\xymatrix{\Frac(A)\otimes_AA'\ar[r]& \Frac(B)\otimes_AA'\\
\Frac(A)\ar[r]\ar[u]& \Frac(B)\ar[u]}.\]
The homomorphism $\Frac(A)\rightarrow \Frac(B)$ is finite, so $\Frac(A)\otimes_AA'\rightarrow \Frac(B)\otimes_AA'$ is finite and flat. For any generic point $\mathfrak{p}\in \Spec(\Frac(A)\otimes_AA')$, we have that
\[[\Frac(B)\colon\Frac(A)] = \sum\limits_{\mathfrak{q}, \varphi(\mathfrak{q})=\mathfrak{p}}[(\Frac(B)\otimes_AA')_{\mathfrak{q}}\colon (\Frac(A)\otimes_AA')_{\mathfrak{p}}]\]
where $\mathfrak{q}$ runs through the minimal ideal of $\Frac(B)\otimes_AA'$ such that $\varphi(\mathfrak{q})=\mathfrak{p}$, and we view $\varphi\colon \Spec(\Frac(B)\otimes_AA')\to \Spec(\Frac(A)\otimes_AA')$. The right-handed side is exactly $d(V,\overline{\{\mathfrak{p}\}}^{V_{\zar}})$  given in the proof of \cref{lemma:welldefinedofdegree} and equal to $\deg(Y^\an/X^\an)$, so (4) holds.

(5) We can assume that $\alpha,\beta$ are prime. Since flat pull-backs preserve proper intersection, by \cref{lemma:properlyintersectislocal}, we know that $\alpha, \beta$ intersect properly if and only if $\alpha^\an, \beta^\an\in Z^*(X^\an)$ intersect properly. The proof of the equality is similar with the proof of (3).
\end{proof}

\section{The category of finite correspondences}

\label{The category of finite correspondences}

In this section, we will define the additive category $\mathrm{Cor}_K$ of finite correspondences of $K$-analytic spaces. We will follow the notation in \cite{ayoub2015motifs} and the idea in \cite[Lecture~1]{mazza2006lecture}.

For the $K$-analytic spaces in this section, we always mean separated, quasi-paracompact (see \cite[Definition~8.12]{bosch2014lecture} for quasi-paracompactness), strictly $K$-analytic spaces, the category of such spaces is exactly the category of separated, quasi-paracompact, $K$-rigid spaces by \cite[Theorem~1.6.1]{berkovich1993etale}.

A $K$-analytic space is said to be \emph{quasi-smooth} if it is geometrically regular at each point, see \cite[Corollary~5.3.5]{ducros2018families}. In particular, a quasi-smooth space is regular.

\begin{definition}
	Let $X$ be a quasi-smooth, connected $K$-analytic space, and $Y$ a $K$-analytic space. An \emph{elementary correspondence} from $X$ to $Y$ is an irreducible closed subset $W$ of $X\times Y$ whose associated integral subspace is finite and surjective over $X$.
	
	For an elementary correspondence from a quasi-smooth non-connected $K$-analytic space $X$ to $Y$, we mean an elementary correspondence from a connected component of $X$ to $Y$.
	
	The group $\Cor_K(X,Y)$ is the free abelian group generated by the elementary correspondences from $X$ to $Y$, so $\Cor_K(X,Y)\subset Z^*(X\times Y)$. An element of $\Cor_K(X,Y)$ are called a \emph{finite correspondence}. By definition, if $X=\coprod\limits_i X_i$ is the decomposition into its connected components, we have $\Cor_K(X,Y)=\bigoplus\limits_i\Cor_K(X_i,Y)$.
\end{definition}

\begin{remark}
	Let $X$ be a quasi-smooth $K$-analytic space, and $Y$ a $K$-analytic space. Every closed subspace $Z$ of $X\times Y$ which is finite and surjective over $X$ determines a finite correspondence $[Z]$ from $X$ to $Y$. Indeed, we consider the case where $X$ is connected (hence irreducible). We can write $[Z]=\sum\limits_in_i[Z_i]$, where $Z_i$ are irreducible component of $Z$ such that $Z_i\rightarrow X$ is surjective, and $n_i$ is the geometric multiplicity of $Z_i$ of $Z$.
\end{remark}

To define the composition of morphism in the category $\Cor_K$, we need the following lemmas.

\begin{lemma}	\label{lemma:imageoffiniteisfinite}
	Let $f\colon T\to T'$ be a morphism of $K$-analytic spaces over another $K$-analytic space $S$. Let $W$ be an irreducible Zariski-closed subset of $T$ which is finite and surjective over $S$. Then $f(W)$ is irreducible, Zariski-closed in $T'$ and finite, surjective over $S$. 
\end{lemma} 
\begin{proof}
	Since $T\rightarrow S$ is separated ($T$ and $S$ are separated), $W\rightarrow S$ is finite, hence proper by \cite[Corollary~3.3.8]{berkovich1990spectral}, we have that $W\rightarrow T'$ is proper, see \cite[\S~9.6.2, Proposition~4]{bosch1984nonarchimedean}. So $f(W)$ is irreducible and Zariski-closed in $T'$. 
	
	We replace $T, T'$ by $W, f(W)$ respectively, so we assume that $T\to S$ is finite and surjective, and $T\to T'$ is proper and surjective. By \cite[Corollary~3.3.8]{berkovich1990spectral}, it remains to show that $T'$ is proper over $S$. Obviously $T' \rightarrow  S$ is quasi-compact since $T\rightarrow T'$ is surjective and $T'\rightarrow S$ quasi-compact. By \cite[Proposition~2.5.8~(iii)]{berkovich1990spectral}, we have 
	\[T=\mathrm{Int}(T/S) = \mathrm{Int}(T/T')\cap f^{-1}(\mathrm{Int}(T'/S))=f^{-1}(\mathrm{Int}(T'/S)),\]
	this implies that $\mathrm{Int}(T'/S) = T'$, i.e. $\partial(T'/S)=\emptyset$. So $T'$ is proper over $S$.
\end{proof}
\begin{remark} \label{graph of morphism}
	Let $f\colon X\to Y$ be a morphism of $K$-analytic spaces with $X$ quasi-smooth, and $\Gamma_f$ the {graph of $f$}, see \cref{def:graph of a morphism}. Then $\Gamma_f$ is a finite correspondence from $X$ to $Y$. Indeed, we can assume that $X$ is connected, hence irreducible. By By the lemma above with $S=T=W=X$, $T'=X\times Y$, we have that $(\id_X,f)(X)$ is irreducible, Zariski-closed in $X\times Y$ and finite, surjective over $X$. By \cref{lemma:zariskiimageforreducedspace}, we have that $\Gamma_f=(\id_X,f)(X)$, this proves our claim. Moreover, as a cycle on $X\times Y$, we have that $[\Gamma_f]=(\id_X,f)_*([X])$
\end{remark}	

\begin{lemma} 	\label{lemma:componentoffiberproductissurjective}
	Let $Z$ be an integral $K$-analytic space, finite and surjective over a normal $K$-analytic space $S$. Then for every morphism $S'\rightarrow S$ with $S'$ connected (resp. irreducible), every connected (resp. irreducible) component of $Z\times_SS'$ is finite and surjective over $S'$.
\end{lemma}
\begin{proof}
	This is in fact an algebraic result from \cite[Proposition~2.17]{voevodsky2000cycles}. We can assume that $S=\mathcal{M}(A), Z=\mathcal{M}(B)$ and $S'=\mathcal{M}(A')$ are affinoid. Since $B$ is finite over $A$, so $B'\coloneqq B\widehat{\otimes}_AA' = B\otimes_AA'$.
	
    Since $A$ is normal, we have that $\Spec(A)$ is geometrically unibranch in sense of \cite[0.23.2.1, 6.15.1]{grothendieck1967egaiv4}. Notice that $\Spec(A)$ is geometrically unibranch in sense of \cite[Definition~2.1.5]{voevodsky2000cycles} by \cite[18.8.15]{grothendieck1967egaiv4}. Notice that $\Spec(B)\rightarrow \Spec(A)$ is finite, surjective with $B$ integral, hence equidimensional of dimension $0$ in sense of \cite[Definition~2.1.2]{voevodsky2000cycles}. By \cite[Proposition~2.1.7~(3), (1)]{voevodsky2000cycles}, $\Spec(B)\rightarrow \Spec(A)$ is universally open in sense of \cite[Definition~2.1.4]{voevodsky2000cycles}. Then $\Spec(B')\rightarrow \Spec(A')$ is open. For every connected component $T=\mathcal{M}(C)$ of $\mathcal{M}(B')$, the morphism $\Spec(C)\rightarrow \Spec(A')$ is open. So $\mathcal{M}(C)\rightarrow \mathcal{M}(A')$ has image that is closed and Zariski-open, which is exactly $\mathcal{M}(A')$ since it is connected. 
	
	For the irreducible case, since $\Spec(B')\rightarrow \Spec(A')$ is equidimensional. Then the image of each irreducible component $\Spec(C)$ of $\Spec(B')$ is $\Spec(A')$. Since the image of $\mathcal{M}(C)$ is a Zariski-closed subspace of $\mathcal{M}(A')$, it must be $\mathcal{M}(A')$.
\end{proof}

\begin{lemma}
	Let $X, Y, Z$ be $K$-analytic spaces. Let $V\subset X\times Y$ and $W\subset Y\times Z$ be integral closed subspace which are finite and surjective over $X$ and $Y$ respectively. Assume that $Y$ is normal. Then $V\times Z$ and $X\times W$ intersect properly in $X\times Y\times Z$, and each irreducible component of the push-forward of the cycle $[V\times Z]\cdot[X\times W]$ on $X\times Z$ is finite and surjective over $X$.
\end{lemma}
\begin{proof}
	Notice that $V\times_YW\hookrightarrow X\times Y\times_YY\times Z\simeq X\times Y\times Z$ is the intersection of $V\times Z$ and $X\times W$ in $X\times Y\times Z$. Then we have the following diagram
	\[\xymatrix{V\times_YW\ar[r]\ar[d] & W\ar[r]\ar[d] & Z\\
		V\ar[r]\ar[d]& Y&\\
		X}.\]
	Since $W\to Y$ is finite and surjective with $Y$ normal, by \cref{lemma:componentoffiberproductissurjective}, each component of $V\times_YW$ is finite and surjective over $V$, so it is also finite and surjective over $X$, and it is of dimension $\dim X$. This implies that $V\times Z$ and $X\times W$ intersect properly in $X\times Y\times Z$. By \cref{lemma:imageoffiniteisfinite}, the image of each irreducible component of $V\times_YW$ in $X\times Z$ is finite and surjective over $X$ (we take $S=X$, $T=V\times_YW$, $T'=X\times Z$ in \cref{lemma:imageoffiniteisfinite}).
\end{proof}

\begin{definition}
	Let $\mathrm{Cor}_K$ be the category defined as follows:
	\begin{itemize}
		\item Objects: the quasi-smooth $K$-analytic spaces;
		\item Morphisms: the finite correspondences $\Cor_K(X,Y)$.
	\end{itemize}
	Given $V\in \Cor_K(X,Y), W\in \Cor_K(Y,Z)$, we define $W\circ V$ as the push-forward of $[V\times Z]\cdot[X\times W]$ on $X\times Z$, which is an element in $\Cor_K(X,Z)$. 
\end{definition}

We need to check that $\Cor_K$ is a category, i.e. the composition is associative and has an identity. This is given by following proposition. 

\begin{proposition} \label{prop:category of finite correspondences}
	Let $X,Y,Z$ be quasi-smooth $K$-analytic spaces, and $\alpha\in \Cor_K(X,Y), \beta\in \Cor_K(Y,Z)$.
	\begin{enumerate}
		\item [(1)] If $\gamma\in \Cor_K(Z,Z')$, then $\gamma\circ(\beta\circ\alpha) = (\gamma\circ\beta)\circ\alpha$.
		\item[(2)] The following statements hold. \begin{enumerate} 
			\item [(2.1)] If $\beta=\Gamma_g$, then $\beta\circ\alpha=(\id_X\times g)_*(\alpha)$.
			\item [(2.2)] If $\alpha=\Gamma_f$, then $\beta\circ\alpha=(f\times \id_Y)_*(\beta)$.
			\item [(2.3)] If $\alpha=\Gamma_f$, $\beta=\Gamma_g$, then $\beta\circ\alpha=\Gamma_{gf}$.
		\end{enumerate}
		In particular, the diagonal $\Delta_X$ is the identity for a quasi-smooth $K$-analytic space $X$.
	\end{enumerate}
\end{proposition}
\begin{proof}
This is from \cref{prop:flatbasechangeofcycles} and \cref{theorem:projectionformula}, see the proof of \cite[Proposition~16.1.1~(a)~(c)]{fulton1998intersection} for the details.
\end{proof}
\begin{remark}
The category $\mathrm{QSm}_K$ of quasi-smooth $K$-analytic spaces is subcategory of $\Cor_K$. 
\end{remark}

\begin{remark} \label{equivalence of category of finite correspondences}
Our definition of $\mathrm{Cor}_K$ coincides with \cite[DEFINITION~2.2.29]{ayoub2015motifs} (notice that rigid spaces in \cite{ayoub2015motifs} are assumed to be separated and to have an admissible covering by countable affinoid domains (in particular quasi-paracompact as G-topological spaces) since \cite[\S~1.2]{ayoub2015motifs}). Indeed, we denote $\mathbf{RigCor}(K)$ the category defined in \cite[DEFINITION~2.2.29]{ayoub2015motifs}, $\mathrm{SmRig}_{/K}$ the category of smooth rigid spaces over $K$ with \'etale topology, and $\mathcal{C}$ the fully faithful subcategory of $\mathrm{Cor}_K$ having the same objects as $\mathbf{RigCor}(K)$. Any $Y\in\Ob(\mathcal{C})$ defines an abelian sheaf on $\mathrm{SmRig}_{/K}$, denoted by $\Z_{\mathrm{tr}}(Y)$, see \cite[\S~2.2.3]{ayoub2015motifs}. Notice that for any $X,Y\in\Ob(\mathcal{C})$, $\Hom_{\mathbf{RigCor}(K)}(X,Y)= \Z_{\mathrm{tr}}(Y)(X)$, see \cite[DEFINITION~2.2.29]{ayoub2015motifs}, and by \cite[PROPOSITION~2.2.28]{ayoub2015motifs}, $\Z_{\mathrm{tr}}(Y)(X)= \Hom_{\mathcal{S}h(\mathrm{SmRig}_{/K})}(\Z_{\mathrm{tr}}(X),\Z_{\mathrm{tr}}(Y))$, where $\mathcal{S}h(\mathrm{SmRig}_{/K})$ the category of abelian sheaves on $\mathrm{SmRig}_{/K}$. It suffices to show that the canonical isomorphism 
$\Hom_{\mathbf{RigCor}(K)}(X,Y)=\Z_{\mathrm{tr}}(Y)(X)\overset{F}{\simeq} \Cor_K(X,Y)$ in \cite[PROPOSITION~2.2.35]{ayoub2015motifs} defines a functor $\mathbf{RigCor}(K)\to\mathcal{C}$, i.e. $F(\mathbbm{1}_X)=\Delta_X$, $F(g\circ f)=F(g)\circ F(f)$. The first equality is easily deduced from the construction of $F$ in the proof of \cite[PROPOSITION~2.2.35]{ayoub2015motifs}. For the second one, let $X,Y, Z\in\Ob(\mathcal{C})$ and $g\colon Y\to Z$ in $\mathbf{RigCor}(K)$, we should show the following diagram commutes
\[\xymatrix{\Z_{\mathrm{tr}}(Y)(X)\ar[r]^{g}\ar[d]_-{F}^-{\rotatebox{90}{$\sim$}}&\Z_{\mathrm{tr}}(Z)(X)\ar[d]^-{F}_-{\rotatebox{90}{$\sim$}}\\
\Cor_K(X,Y)\ar[r]^-{F(g)}&\Cor_K(X,Z)}.\]
Since $\Z_{\mathrm{tr}}(Y)\overset{F}{\simeq} \Cor_K(-,Y)$ is an isomorphism of sheaves on $\mathrm{SmRig}_{/K}$, see the proof of \cite[PROPOSITION~2.2.35]{ayoub2015motifs}, we can assume that $X$ is affinoid. In this case, we have $\Z_{\mathrm{tr}}(Y)(X)$ is the free abelian group generated by $\Hom_{\mathrm{SmRig}_{/K}}(X,Y)$, and $F(f)=\Gamma_f$ for $f\in\Hom_{\mathrm{SmRig}_{/K}}(X,Y)$. If $Y$ is affinoid, then the diagram commutes by \cref{prop:category of finite correspondences}~(2.3), i.e. $g=F^{-1}\circ F(g)\circ F$ in $\Hom_{\mathcal{S}h(\mathrm{SmRig}_{/K})}(\Z_{\mathrm{tr}}(Y),\Z_{\mathrm{tr}}(Z)) = \Z_{\mathrm{tr}}(Z)(Y)$. For general case, we cover $Y$ with admissible covering $\{Y_i\}_{i\in I}$ by affinoid domains, since $F$ is canonical, we have that $g=F^{-1}\circ F(g)\circ F$ holds on $Y_i$. Then the $g=F^{-1}\circ F(g)\circ F$ holds in $\Z_{\mathrm{tr}}(Z)(Y)$ since $\Z_{\mathrm{tr}}(Z)$ is a sheaf. This proves the statement.
\end{remark}

Following the idea in \cite{bloch1986algebraic}, we can define higher Chow groups $\CH^n(X,s)$ for quasi-smooth $K$-analytic spaces. By GAGA principle, such definition will coincide with the one for proper varieties. On the other hand, the higher Chow groups is also defined in \cite[Introduction g\'en\'erale]{ayoub2015motifs} using motives of analytic spaces. It is natural to expect there is a close connection between these two and higher Chow groups have similar properties as in the case of algebraic varieties. 


	



	\section*{Acknowledgements}
	The author would like to thank his host professor, Yigeng Zhao for his encouragement, support and valuable suggestions. He would also like to thank Antoine Ducros, Walter Gubler and Michael Temkin for their patience and answering questions during his study of Berkovich spaces. Also the author would like to thank the referee for carefully reading this paper, pointing out some mistakes and useful suggestions. This research is supported by postdoctoral research grant.

\end{document}